\documentclass[10pt]{article}
\usepackage{hyperref}
\usepackage[font=small,labelfont=bf]{caption}
\usepackage{import}
\usepackage{pdfpages}
\usepackage{transparent}
\usepackage{xcolor}
\usepackage{authblk}
\usepackage{subcaption}
\usepackage{tikz}
\usepackage{amsthm}
\usepackage[UKenglish]{babel}

\usepackage{amsmath,amsfonts,amssymb}
\usepackage{graphicx}
\usepackage{orcidlink}
\usepackage{subcaption}
\usepackage{tikz}
\usepackage{import}
\usepackage{xifthen}
\usepackage{pdfpages}
\usepackage{transparent}

\usepackage{bm}
\usepackage{csquotes}

\hypersetup{
    colorlinks,
    linkcolor={red!50!black},
    citecolor={blue!50!black},
    urlcolor={blue!80!black}
}
\newtheorem{Definition}{Definition}
\newtheorem{Remark}{Remark}
\newtheorem{Proposition}{Proposition}
\newtheorem{Corollary}{Corollary}
\newtheorem{Theorem}{Theorem}
\newtheorem{Lemma}{Lemma}

\DeclareMathOperator{\sgn}{sgn}

\DeclareMathOperator{\id}{Id}

\DeclareMathOperator{\fix}{Fix}

\newcommand{\cC}{\mathcal{C}}

\usepackage{microtype}

\usepackage{mathtools}

\mathtoolsset{showonlyrefs=true}

\title{Nonlinear Diffusion on Networks: Perturbations and Consensus Dynamics}
    
\date{}

\author[$\dagger$]{Riccardo Bonetto\orcidlink{0000-0001-8075-6147}}
\author[$\dagger$]{Hildeberto Jardón Kojakhmetov\orcidlink{0000-0001-8708-7409}}

\begin{document}
\affil[$\dagger$]{University of Groningen — Bernoulli Institute for Mathematics, Computer Science and Artificial Intelligence; Nijenborgh 9, 9747AG, Groningen, The Netherlands}
\maketitle




\begin{abstract}
    In this paper, we study a class of equations representing nonlinear diffusion on networks. A particular instance of our model could be seen as a network equivalent of the porous-medium equation. We are interested in studying perturbations of such a system and describing the consensus dynamics. The nonlinearity of the equations gives rise to potentially intricate structures of equilibria that could intersect the consensus space, creating singularities. For the unperturbed case, we characterize the sets of equilibria by exploiting the symmetries under group transformations of the nonlinear vector field. Under small perturbations,  we obtain a slow-fast system. Thus, we analyze the slow-fast dynamics near the singularities on the consensus space. The analysis at this stage is carried out for complete networks, allowing a detailed characterization of the system. We provide a linear approximation of the intersecting branches of equilibria at the singular points; as a consequence, we show that, generically, the singularities on the consensus space turn out to be \emph{transcritical}. We prove under local assumptions the existence of canard solutions. For generic graph structures, assuming more strict conditions on the perturbation, we prove the existence of a maximal canard, which coincides with the consensus subspace. In addition, we validate by numerical simulations the principal findings of our main theory, extending the study to non-complete graphs. Moreover, we show how the delayed loss of stability associated with the canards induces transient spatio-temporal patterns.  
\end{abstract}

\maketitle
\section{Introduction}

Diffusion is certainly among the most renowned subjects in dynamical systems. The ordinary differential equation $\dot{\bm{x}} = - L \bm{x}$, where $L$ is the graph Laplacian, describes a (linear) diffusion process on a network. In fact, the Laplacian matrix plays the role of a discrete counterpart of the
Laplacian differential operator, $\Delta$, present for example in the heat
equation, $\dot u = \Delta u$.
A lot of effort has been made to precisely analyze the behavior of Laplacian
systems, especially for the linear case
\cite{veerman2020primer,VEERMAN2019184,1470239}. More recently, nonlinear
extensions of the Laplacian dynamics have been explored
\cite{5531534,6329411,6859170,doi:10.1137/20M1376844,Jardon-Kojakhmetov2020}.

In the present paper, we aim to analyze a class of nonlinear Laplacian systems defined by
    \begin{equation}\label{eq:ALF_analogy}
        \dot{\bm{x}} = - L F(\bm{x}) + \epsilon H(\bm{x}, \Lambda) ,
    \end{equation}
    where $\bm{x}=\bm{x}(t)\in\mathbb{R}^n$, $F(\bm{x})$ is a nonlinear vector field, $0<\epsilon \ll 1$ is a small parameter, and  $H(\bm{x}, \Lambda)$ is a perturbation with $\Lambda\in\mathbb R^p$ an extra set of parameters; more precise definitions follow in the next sections. Equation \eqref{eq:ALF_analogy} describes a nonlinear diffusion process on a network. In the literature, such systems are also known as \emph{Absolute Laplacian Flows} (ALFs) \cite{5531534,6329411}.
Similarities with the continuum case can be found by looking at drift-diffusion
equations \cite{RevModPhys.15.1}, in particular with the porous-medium equation, $\dot u = \Delta
        u^m$, \cite{vazquez2007porous}. The porous-medium equation has several
    potential applications, for example: flow of gases \cite{leibenzon1930motion,
        muskat1938flow}, theoretical biology \cite{simpson2011models}, and
    particle-based models \cite{gurney1975regulation} (where the power considered
    is $m=2$). A network model for the porous-medium equation has been proposed in
    \cite{falco2022random}; similar systems have also been studied in
    \cite{carletti2020nonlinear} and \cite{rahimabadi2023extended}.
Let us moreover notice that by considering the rewriting $F(\bm{x}) =:
    K(\bm{x}) \bm{x}$ one can interpret the matrix $K(\bm{x})$ as a state-dependent
diffusion coefficient, leading, in turn, to another interpretation in terms of
the network structure. In fact, we can define a state-dependent Laplacian,
$L_{\textup{K}}(\bm{x}):=L K(\bm{x})$, which can be seen either as the
Laplacian of a directed graph under the definition of out-degree Laplacian
\cite{AHMADIZADEH2017281,doi:10.1063/1.5139137}, or as a particular case of the
Laplacian tensor for consensus models on hypergraphs \cite{Sahasrabuddhe_2021}.
Nonlinear diffusion processes have also been considered in discrete time
systems, in particular for lattice networks \cite{peng2023rich}.

Within the context of Laplacian dynamics the diffusion of `information' among
the agents can lead the system to reach a common state, referred to as
\emph{consensus}. As such, Laplacian systems find a remarkable application in
consensus problems \cite{1470239,1239709,9483355}, which have been extensively
studied for their importance in several branches of science, e.g.,
synchronization, opinion dynamics, coordination of robots, rendezvous problems,
among many others.

One significant challenge in studying networked dynamical systems is the high dimensionality of the problem. Furthermore, the nonlinearities inherent in such models contribute to an even greater degree of complexity. The contributions of this paper address such challenges in particular contexts where the model we consider is a nonlinear version of the well-known consensus dynamics. For complete graphs, under appropriate conditions accounting for the symmetries of the network system, we provide a reduction to a planar system, see Lemma \ref{lm:pert_assumption}. Hence, in this particular setup, a large dimensional nonlinear system can be studied with low-dimensional techniques of dynamical systems. For example, by exploiting the aforementioned reduction, we prove that, generically, the singularities along the consensus space are transcritical; see Proposition \ref{prop:tangents}. Later, for nondegenerate transcritical singularities, we provide conditions for the existence of ``canard solutions'' in Proposition \ref{prop:canard}. These canards are relevant in the context of perturbation problems since they are solutions that are not expected from zeroth-order approximations. Although the previous results are for complete graphs, we also notice that, under some conditions, the canard condition can be provided for more general graph structures; see Proposition \ref{prop:canard_generic}. Finally, the insights that we gain from our analysis are used in some numerical experiments showing some intriguing phenomena that we associate with the canards; see Section \ref{sec:spatio-temporal}.

The paper is structured as follows: In Section \ref{sec:def_and_pre}, we
define the (unperturbed) ALF and exploit some general properties. Section
\ref{sec:symmetries} is dedicated to understanding the role of symmetries in
ALFs; for a brief introduction to the main concepts related to symmetries see
Appendix \ref{app:sym}. In Section \ref{sec:perturbations},
we introduce a perturbation term leading to \eqref{eq:ALF_analogy}. As we will
see, the perturbations, in general, give rise to a slow-fast behavior of the
system. A brief summary of the main ideas of singular perturbation theory
useful for the main analysis can be found in Appendix \ref{app:singular}.
The slow-fast consensus dynamics of ALFs with complete graph structure is
studied in Section \ref{sec:consensus}, where we also provide
    conditions under which systems with arbitrary graph structure admit a maximal
    canard. In Section \ref{sec:ex_and_sim}, we explore numerically different
network structures, spatio-temporal configurations, and we describe the effects
of canard conditions in different setups. A brief discussion and the
conclusions are drawn in Section \ref{sec:conclusions}.


\section{Definitions and preliminaries}\label{sec:def_and_pre}

From a mathematical point of view, an undirected network structure can be described by a
graph, $\mathcal{G}=\{ \mathcal{V}, \mathcal{E} \}$, which consists of a finite
set of nodes (or vertices) $ \mathcal{V}=\{ 1,...,n \}$, and a set of edges
$e_{ij} \in \mathcal{E}$, where each $e_{ij}$ represents a connection between
the node $j$ and $i$ \cite{godsil2001algebraic}. Furthermore, if we assign a
set of weights $\mathcal{W}$ to the graph $\mathcal{G}$, then we obtain a
weighted graph, $\mathcal{G}_\textup{w}=\{ \mathcal{V}, \mathcal{E},
    \mathcal{W} \}$. We assume (unless otherwise stated) that the weights, $w_{ij}
    \in \mathcal{W}$, are positive reals. Such a choice is natural both for the
algebraic properties that follow from it \cite{beineke2004topics}, and for the
modeling perspective, since often the weights represent some measure of
distance or strengths of the coupling \cite{newman2018networks}. In other
contexts, negative weights are also considered \cite{doi:10.1137/17M1134172},
for example in models representing inhibition or antagonism; in this paper
we do not consider such situations. The edges and the weights are in one-to-one
correspondence, in fact, we assign to each edge $e_{ij}$ a weight $w_{ij}$. Throughout this paper, we consider only \emph{undirected} graphs. Therefore, we have that the weights are symmetric, that is $w_{ij} = w_{ji}$.

\begin{Definition}
    A \emph{simple graph} is an undirected graph without self-loops or multiple edges.
\end{Definition}
Unless otherwise stated, we consider simple graphs. For weighted graphs, we assume a simple unweighted structure, i.e., $\mathcal{G}_\textup{w} \setminus \mathcal{W}$ is a simple graph. Moreover, we make no nomenclature distinction between graphs and weighted graphs with unit weights, because the algebraic structures in such cases are exactly the same.

We now recall a few well-known concepts of graph theory
\cite{newman2018networks,biggs1993algebraic,godsil2001algebraic}. A \emph{path}
is a sequence of edges that joins a sequence of distinct vertices. A
\emph{connected component} of a graph is a sub-graph such that every pair of
nodes is connected by a path. A relevant class of graphs is given by the
complete graphs.

\begin{Definition}
    A \emph{complete graph} with $n$ nodes, denoted by $K_n$, is a graph such that for all pairs of nodes there is an edge connecting them, i.e., $\forall i,j \in \mathcal{V}$, with $i \neq j$, $\exists \ e_{ij} \in \mathcal{E}$.
\end{Definition}

Given a graph it is possible to define some algebraic structures on it. We
recall here the ones that are useful for the purpose of the paper.

\begin{itemize}
    \item The adjacency matrix, $A$, of a graph is the $n \times n$ matrix with
          components $A_{ij} = w_{ij}$ if $e_{ij} \in \mathcal{E}$, and $A_{ij} = 0$ if
          $e_{ij} \notin \mathcal{E}$.
    \item The degree matrix, $\Delta$, of a graph is the $n \times n$ diagonal matrix
          with (diagonal) components $\Delta_{ii} = \sum_j w_{ij}$.
    \item The Laplacian matrix, $L$, of a graph is the $n \times n$ matrix defined by $L
              := \Delta - A$.
\end{itemize}

For simple graphs, let us recall that the Laplacian matrix, $L$, is symmetric,
degenerate and semi-positive definite; moreover, the following proposition
holds \cite{beineke2004topics}.

\begin{Proposition}\label{prop:mult_and_vec}
    The algebraic multiplicity of the $0$ eigenvalue of $L$, $\mu_a(0)$, is equal to the number of connected components of the graph. Furthermore, the vector $\mathbf{1}=(1, ..., 1)^\intercal$ is an eigenvector of $L$ with eigenvalue $0$, i.e., $\mathbf{1} \in \textup{ker}(L)$.
\end{Proposition}

We now introduce the class of systems we are going to investigate in the
present paper. Let us assign to each node of a graph, $i \in \mathcal{V}$, a
state $x_i \in \mathbb{R}$ and a smooth function
    \begin{equation}\label{eq:response_function}
        \begin{aligned}
            f : \mathbb{R} & \to \mathbb{R}   \\
            x_i            & \mapsto f(x_i) ,
        \end{aligned}
    \end{equation}
    which we call the \emph{response function}. So, given the set of nodes $\mathcal{V}$, we obtain a state vector $\bm{x} =(x_1, ..., x_n)^\intercal$, and a response vector field $F(\bm{x})= \left(f(x_1), \dots, f(x_n) \right)^\intercal$. An Absolute Laplacian Flow \cite{5531534,6329411} is defined by the equation
\begin{equation}\label{eq:LDN}
    \dot{\bm{x}} = - L F(\bm{x}) .
\end{equation}
\begin{Remark}
    Due to Proposition \ref{prop:mult_and_vec}, the Laplacian decomposes into a direct sum $L = L^{(1)} \oplus \dots \oplus L^{(\mu_a(0))}$, and, therefore, an ALF decomposes in $\mu_a(0)$ independent systems. So, in general, we consider connected graphs.
\end{Remark}

Note that if we choose the response vector field to be the identity, i.e.,
$F(\bm{x}) = \bm{x}$, then Eq \eqref{eq:LDN} becomes the widely studied (linear)
Laplacian dynamics, $\dot{\bm{x}} = - L \bm{x}$,
\cite{veerman2020primer,VEERMAN2019184,1470239}. Indeed, the ALF is a direct
nonlinear generalization of the linear Laplacian dynamics
\cite{5531534,6329411,6859170,doi:10.1137/20M1376844}. ALFs have received
considerably less attention than the linear Laplacian dynamics, although they have strong relations with consensus problems. For example, ALFs have
been considered as models for nonlinear communications protocols in
\cite{8421087,NOSRATI20122262}.

An important observation is that, in contrast to linear Laplacian dynamics and depending on the response function, ALFs may exhibit bifurcations of the consensus state, see \eqref{eq:consensus}. Therefore, our objective is to characterize the dynamic behavior near consensus  for such systems under small perturbations.


\subsection{General properties}\label{sec:general_properties}

ALFs have some generic properties derived from the symmetric and
degenerate structure of the Laplacian matrix. First, notice that since
$\bm1\in\ker L$, then for any $h\in \cC^m(\mathbb{R}^n,\mathbb{R})$,
    $m\geq3$, one has that $L(F(\bm{x})+h(\bm{x})\bm{1})=LF(\bm{x})$, meaning that
\emph{ALFs are invariant under response field transformations $F(\bm{x})
        \mapsto F(\bm{x}) + h(\bm{x}) \mathbf{1}$}. As a consequence, we have that the
response vector field is not unique: in fact, there exists an infinite set of
vector fields leading to the same dynamics. This freedom can be used in
practice to choose the \enquote{simplest} response field. Such a property
reminds in some way of a gauge freedom \cite{landau2013classical}, where
one can choose among a class of response fields $F^{(h)}(\bm{x}) = F(\bm{x}) +
    h(\bm{x}) \mathbf{1}$ without affecting the evolution equation \eqref{eq:LDN}.
Another important property of ALFs is the existence of a constant of motion.

\begin{Lemma}\label{lm:constant}
        Given an ALF \eqref{eq:LDN}, there exists a constant of motion $k$,
        \begin{equation}\label{eq:constant}
            k = \langle \mathbf{1} , \bm{x} \rangle ,
        \end{equation}
        where $\langle \cdot  , \cdot \rangle$ is the usual Euclidean scalar product in $\mathbb{R}^n$.
    \end{Lemma}

\begin{proof}
    The time evolution of $k$ is given by $\dot{k}=\langle \mathbf{1} , \dot{\bm{x}} \rangle $. Then we have $\langle \mathbf{1} , \dot{\bm{x}} \rangle = -  \mathbf{1}^\intercal L F(\bm{x}) = -  (L \mathbf{1})^\intercal F(\bm{x}) = 0 $. Thus, $k$ is a constant (in time) value.
\end{proof}

It is well known that the constant of motion $k$ is related to the arithmetic mean of the states, $k=n \langle \bm{x} \rangle$, where $\langle \bm{x} \rangle := \langle \mathbf{1},\bm{x} \rangle /n$. Thus, we can interpret the presence of the constant of motion $k$ as a conservation law for $\langle \bm{x} \rangle$. More generally, every quantity of the form $c k$, with $c$ a constant, is a constant of motion for an ALF.

\begin{Lemma}
        The set of equilibria of an ALF \eqref{eq:LDN} is given by
        \begin{equation}\label{eq:equilibria_ldn}
            E := \{ \bm{x} \in \mathbb{R}^n \ | \ f(x_1)= \dots = f(x_n) \} .
        \end{equation}
    \end{Lemma}

\begin{proof}
    In order to have an equilibrium for \eqref{eq:LDN} we need  $F(\bm{x}) \in \ker(L)$, which means $f(x_1)= \dots = f(x_n)$, implying in turn that $\bm{x}$ is an equilibrium if and only if $\bm{x} \in E$.
\end{proof}

\begin{Corollary}\label{cor:consensus}
    The set
    \begin{equation}\label{eq:consensus}
        C := \{ \bm{x} \in \mathbb{R}^n \ | \ x_1 = \dots = x_n \}
    \end{equation}
    is a subset of equilibria of system \eqref{eq:LDN}, i.e., $C \subseteq E$.
\end{Corollary}
The set $C$ is the so-called \emph{consensus} space \cite{Jardon-Kojakhmetov2020,1470239,1239709}.
As we already mentioned, one of our goals is to characterize the behavior near
the consensus space \eqref{eq:consensus}. Our characterization aims to be
qualitative, in the sense of \emph{topological equivalence}
\cite{guckenheimer1983nonlinear,wiggins2003introduction}. 
    \begin{Lemma}
        The consensus space can be parametrized by the constant k, i.e.,
        \begin{equation}
            C = \{ \bm{x} \in \mathbb{R}^n \ | \ x_1 = \dots = x_n = \frac{k}{n}, \ k \in \mathbb{R}\}.
        \end{equation}
    \end{Lemma}
    \begin{proof}
        Follows from straightforward computations.
    \end{proof}
    From the previous lemma, we denote a parametrized point in $C$ by $\bm{x}^*_k = (x_k^*, \dots, x_k^*)^\intercal$, where
    \begin{equation}\label{eq:equilibrium}
        x_k^*=\frac{k}{n}.
    \end{equation}
    Notice that, since the system has a constant of motion, the evolution is actually restricted to a family of $(n-1)$-dimensional hyperplanes
    \begin{equation}
        P_k := \{ x \in \mathbb{R}^n  \ | \  \langle \mathbf{1} , \bm{x} \rangle = k \} .
    \end{equation}
    It is known that for linear Laplacian dynamics, systems associated with weighted graphs are topologically equivalent to the corresponding unweighted ones \cite{doi:10.1137/130913973}. An important observation is that ALFs on weighted and unweighted graphs are equivalent near consensus on $P_k$.

\begin{Definition}
    Two vector fields, $X, Y$, are said to be \emph{topologically equivalent} if there exists a homeomorphism which takes orbits of $X$ to orbits of $Y$, preserving directions but not necessarily time parametrisation.
\end{Definition}

\begin{Proposition}\label{prop:equivalence}
     For any $f$ and $k$ fixed, in a small neighborhood of $\bm{x}^*_k \in C$ and restricted to $P_k$, the weighted, $\dot{\bm{x}}=-L_\textup{w}F(\bm{x})$, and the unweighted, $\dot{\bm{x}}=-LF(\bm{x})$, ALFs are topologically equivalent.
\end{Proposition}
\begin{proof}   
        Let $L_\textup{w}$ be the Laplacian of a weighted graph $\mathcal{G}_\textup{w}$. Then the Jacobian of the associated ALF restricted to the consensus space is given by $ - L_\textup{w} \textup{D}F(\bm{x}^*_k) $, where $\textup{D}F(\bm{x}^*_k)$ is the Jacobian of $F(\bm{x})$ evaluated on the consensus set parametrized by $k$. On the other hand, by setting all the weights to one on $\mathcal{G}_\textup{w}$ we obtain the simple graph $\mathcal{G}$. Thus, for the ALF with unweighted graph $\mathcal{G}$ the Jacobian is $- L \textup{D}F(\bm{x}^*_k)$, where $L$ is the Laplacian matrix associated to $\mathcal{G}$. We notice that $ \textup{D}F(\bm{x}^*_k) = \id \textup{d}_xf(x^*_k)$, where $\textup{d}_x f$ is a short-hand notation for the derivative of $f$ with respect to $x$, and $\id$ is the $n\times n$ identity matrix. So, the two matrices to compare are $- \textup{d}_xf(x^*_k) L_\textup{w}$  and $-  \textup{d}_xf(x^*_k) L$. We remark that $- \textup{d}_xf(x^*_k)$ is a common scalar factor and that the matrices $L_\textup{w}$ and $L$ do not depend on $x^*_k$\\
        Let us recall that the Laplacian matrix is a semi-positive definite matrix
        \cite{beineke2004topics}, which means that its eigenvalues are all greater than
        or equal to zero. Such a property holds for both simple graphs and
        positive-weighted graphs. Moreover, since in our context $L_\textup{w}$ and $L$
        are symmetric, we can transform (by a similarity transformation) the Jacobians
        to the diagonal forms
        \begin{equation}
            -  \textup{d}_xf(x^*) \text{diag}(0, \mu_1, \dots, \mu_{n-1}) \quad \text{and} \quad -  \textup{d}_xf(x^*) \text{diag}(0,\lambda_1, \dots, \lambda_{n-1}) ,
        \end{equation}
        where $\{0, \mu_1, \dots, \mu_{n-1}\}$ are the eigenvalues of $L_\textup{w}$, and $\{0,\lambda_1, \dots, \lambda_{n-1}\}$ are the eigenvalues of $L$; in both cases the non-zero eigenvalues are positive. Moreover, since the $0$ eigenvalue is associated to the one-dimensional set of equilibria $C$ (Proposition \ref{prop:mult_and_vec} and Corollary \ref{cor:consensus}), we can reduce the analysis to the hyperplane $P_k$, which leads to the $(n-1)$-dimensional diagonal matrices
        \begin{equation}\label{eq:two_systems}
            -  \textup{d}_xf(x^*_k) \text{diag}( \mu_1, \dots, \mu_{n-1}) \quad \text{and} \quad -  \textup{d}_xf(x^*_k) \text{diag}(\lambda_1, \dots, \lambda_{n-1}) ,
        \end{equation}
        where $x^*_k$ is now a parameter. Whenever the common term $\textup{d}_xf(x^*_k)$ is nonzero, we have that both matrices have eigenvalues with the same sign. Therefore, topological equivalence follows from the Hartman-Grobman Theorem. 
        Actually, it is possible to explicitly compute the homeomorphism between the two systems described by Eq \eqref{eq:two_systems}. Since the matrices are diagonal, effectively, we compute $n-1$ homeomorphisms between one-dimensional systems. The flows of the $i$-th components are respectively
        \begin{equation}
           x_{\mu_i}(t, x_0):= x_0 \exp{\left(-  \textup{d}_xf(x^*_k) \mu_i t\right)}  \quad \text{and} \quad  x_{\lambda_i}(t, x_0):= x_0 \exp{\left(-  \textup{d}_xf(x^*_k) \lambda_i t\right)}.
        \end{equation}
        Thus, the homeomorphism, $z:\mathbb{R} \to \mathbb{R}$, such that $z(  x_{\mu_i}(t, x_0) ) = x_{\lambda_i}(t, z(x_0) )$,  is given by
        \begin{equation}\label{eq:homeo}
            z(x) = \begin{cases}
                x^\frac{\lambda_i}{\mu_i} &\quad \text{if} \quad x \geq 0  ,\\
                - |x|^\frac{\lambda_i}{\mu_i} &\quad \text{if} \quad x < 0 .
            \end{cases}
        \end{equation}
        It is straightforward to check that Eq \eqref{eq:homeo} is a homeomorphism of the real line, as it is also proven in \cite{hirsch2013differential}. Notice that the transformation \eqref{eq:homeo} does not depend on the common factor of the systems \eqref{eq:two_systems}, and, therefore, it holds also when $-  \textup{d}_xf(x^*_k) = 0$.
\end{proof}

\begin{Remark}
    Having in mind Proposition \ref{prop:equivalence}, from now on we assume that ALFs have all weights equal to one, i.e., they are defined by a simple (unweighted) graph $\mathcal{G}$.
\end{Remark}

Notice that the topological equivalence described by Proposition \ref{prop:equivalence}, defines a local equivalence close to the consensus space. However, in what follows the assumption of a simple graph structure is global. This is motivated by the later analysis that is dedicated to the consensus space, which is the focus of the paper. The reader should keep in mind that away of the consensus space the equivalence is not proven, and, therefore, does not apply in general.
The results we stated until this point rely mainly on the algebraic properties
of the Laplacian. For simple graphs, together with Proposition
    \ref{prop:equivalence}, it is particularly useful to exploit the symmetry
properties induced by the graph, which in turn can provide insightful
information for the ALF.


\section{Symmetries}\label{sec:symmetries}
In the following sections, we are going to look at how the symmetry
properties of the response function and of the graph affect the
equilibria and the dynamics of ALFs. Some fundamental notions of group theory
and its applications are contained in Appendix \ref{app:sym}.


\subsection{Equilibria of ALFs}\label{sec:equilibria_homogeneous}

Corollary \ref{cor:consensus} shows that the consensus space is always a subset of equilibria for \eqref{eq:LDN}. Now, we show that, starting from the consensus space, the group transformations leaving the response function invariant induce new equilibria.

\begin{Proposition}\label{prop:invariance}
    Let $f$ be a $\Gamma$-invariant response function. Then, the group $\Gamma_n := \underbrace{\Gamma\times\cdots\times\Gamma}_{n\textnormal{-times}}$ maps the equilibria coinciding with the consensus space, $C$, to equilibria, i.e.,
    \begin{equation}
        \Gamma_n (C) \subseteq E.
    \end{equation}
\end{Proposition}

    \begin{proof}
        Let us consider $-L F( \Gamma_n \bm{x}^*)$, where $\bm{x}^* \in C$. We recall that by definition every component of the vector field $F$ has the same response function $f$. Also, each response function is invariant under $\Gamma$, i.e., $f(\Gamma x) = f(x)$. So we have that $F( \Gamma_n \bm{x}^*) = (f(\Gamma x_1^*), \dots, f(\Gamma x_n^*))^\intercal =(f(x_1^*), \dots, f( x_n^*))^\intercal =  F( \bm{x}^*)$. Therefore $-L F( \Gamma_n \bm{x}^*)=-L F( \bm{x}^*) = 0$.
    \end{proof}

\begin{Remark}
    From Proposition \ref{prop:invariance} it follows that any subgroup of $\Gamma_n$ maps  $C$ to equilibria.
\end{Remark}

The invariance of the response function leads to a richer structure of the
equilibria, similarly to what happens for equivariant bifurcations
\cite{golubitsky2012singularities,chossat2000methods}. Such a richer structure
can give rise, for example, to \emph{bipartite consensus}
\cite{6858991,6329411,7917352} where the agents converge to two separate
clusters of consensus. A simple exemplification of the appearance of
bipartite consensus in ALFs is given by considering the response function $f(x)
    =x^2$. It is clear that the symmetry group for such a function is
$\mathbb{Z}_2$, in particular we are considering the representation $\{ 1,-1
    \}$. From Proposition \ref{prop:invariance} we have that the set of
equilibria consists of all the possible combinations $\pm x_1 = \pm x_2 = \dots
    = \pm x_n$, and, therefore, we have bipartite consensus among the potential
equilibria of the system. Let us notice that the linear set-up for
bipartite consensus requires negative weights on the network structure in order
to obtain such a result, while for ALFs bipartite consensus follows from
nonlinearity, and symmetry properties of the response function. It is
interesting to notice that the lines of bipartite consensus, $\pm x_1 = \pm x_2
    = \dots = \pm x_n$, arising from generic even response functions are saddles;
this property follows from straightforward computations of the Jacobian.


\subsection{Equivariant ALFs}\label{sec:equi_ldn}

Since we are considering systems with a network structure, the symmetry group
associated with the equivariance properties, i.e., the automorphism group of the graph, denoted by $\textup{Aut} (\mathcal{G})$, is a subgroup of the symmetric group over $n$ elements, denoted by $\mathfrak{S}_n$. Moreover, if an ALF is $ \mathfrak{S}_n \supseteq \Gamma$-equivariant, it
follows that it is also equivariant under any subgroup $\Sigma \subset \Gamma$.
Therefore, $\fix(\Sigma)$ is also a flow-invariant set, and
$\fix(\Gamma)\subseteq \fix(\Sigma)$. From these arguments, if we are
considering a $\Gamma$-equivariant ALF, the analysis of the dynamics can be
reduced to the fixed-point spaces $\fix(\Sigma)$, for all subgroups $\Sigma$ of $\Gamma$.
Considering that the phase space of an ALF is $\mathbb{R}^n$, when acting on a state, $\bm{x}$, we use the \emph{standard/permutation representation} of $\mathfrak{S}_n$ in $\mathbb{R}^n$. An element of the standard representation is  called a permutation matrix.

\begin{Definition}
    A \emph{permutation matrix} $\sigma \in \mathfrak{S}_n$ is an $n \times n$ matrix with components
    \begin{equation}
        \sigma_{ij} =
        \begin{cases}
            1 & \quad \text{if $\sigma(i)=j$}, \\
            0 & \quad\text{otherwise} ,              \\
        \end{cases}
    \end{equation}
    where the expression $\sigma(i)=j$ is the abstract action of the group element $\sigma$ on an element of the vertex set $\mathcal{V}=\{1, \dots, n\}$, meaning that the component $\sigma_{ij}$ is $1$ if the vertex $i$ is permuted with the vertex $j$, and is zero otherwise. 
\end{Definition}
From now on, if not otherwise explicitly mentioned, when we write the action of a permutation group, we implicitly consider the standard representation.

\begin{Proposition}\label{prop:Sn_equivariance}
        $\mathcal{G}$-ALFs are $\textup{Aut}(\mathcal{G})$-equivariant, where $\textup{Aut}(\mathcal{G})$ is a subgroup of $\mathfrak{S}_n$ in the standard representation.
    \end{Proposition}
    \begin{proof}
        Let $\sigma \in \textup{Aut}(\mathcal{G})$ be an $n \times n$ permutation matrix. We start by applying a permutation to the vector field, $\sigma L F(\bm{x})$. Let us notice that $ [L, \sigma] = 0$ for all $\sigma \in \textup{Aut}(\mathcal{G})$, where $[ \cdot , \cdot ]$ is the commutator, as it follows from the definition of symmetry for a graph (Definition \ref{def:aut_g}, Appendix \ref{app:sym}). Thus, we have that $\sigma L F(\bm{x})= L \sigma F(\bm{x})$. Notice that a permutation applied to the vector field $F(\bm{x}) = (f(x_1), \dots, f(x_n))^\intercal$ exchanges the order of the components, but since the functions are indistinguishable we have that $\sigma F(\bm{x}) = F(\sigma \bm{x})$. Therefore, $\sigma L F(\bm{x}) = L F(\sigma \bm{x})$, which proves the statement.
    \end{proof}

\begin{Remark}
    Note that the action of $ \textup{Aut}(\mathcal{G})\subseteq\mathfrak{S}_n$ on the (nonlinear) vector field $F(\bm{x})$ is linear.
\end{Remark}

\begin{Proposition}\label{prop:con_sym}
    Let $\textup{Aut}(\mathcal{G})$ be an irreducible subgroup of $\mathfrak{S}_n$ in the standard representation. If $\textup{Ord}( \textup{Aut}(\mathcal{G})) \geq n$ then $\fix(\textup{Aut}(\mathcal{G})) = C$.
\end{Proposition}

\begin{proof}
    Let $\textup{Ord}( \textup{Aut}(\mathcal{G})) =k \geq n$. The fixed-point space is obtained by considering the equations
        \begin{equation}\label{eq:eq_fix}
            \sigma^{(1)} \bm{x} = \dots =\sigma^{(k)} \bm{x} ,
        \end{equation}
        where $\sigma^{(s)} \in \textup{Aut}(\mathcal{G})$, $s=1, \dots, k$. Suppose, by contradiction, that Eq \eqref{eq:eq_fix} does not fix all the components to the same value, but $m<n$ components are fixed to one common value and the other $n-m$ components are fixed to a different common value, i.e., equation \eqref{eq:eq_fix} leads to
        \begin{align}
            x_1     & =\dots =x_m ,                \\
            x_{m+1} & = \dots = x_n .
        \end{align}
        As a consequence, we have that the permutation matrices $\sigma^{(s)}$ can be written as follows
        \begin{equation}
            \begin{pmatrix}
                \sigma^{(s)}_{11} & \dots  & \sigma^{(s)}_{1m} & 0                          & \dots  & 0                      \\
                \vdots            & \ddots & \vdots            & \vdots                     & \ddots & \vdots                 \\
                \sigma^{(s)}_{m1} & \dots  & \sigma^{(s)}_{mm} & 0                          & \dots  & 0                      \\
                0                 & \dots  & 0                 & \sigma^{(s)}_{(m+1) (m+1)} & \dots  & \sigma^{(s)}_{(m+1) n} \\
                \vdots            & \ddots & \vdots            & \vdots                     & \ddots & \vdots                 \\
                0                 & \dots  & 0                 & \sigma^{(s)}_{n (m+1)}     & \dots  & \sigma^{(s)}_{n n}
            \end{pmatrix} ,
        \end{equation}
        which is a block form. This would imply that the representation is decomposable and therefore not irreducible, which is a contradiction. The case where the components are fixed to a number of {distinct values} $l>2$ is analogous, and would give rise to block matrices with $l$ blocks. {In turn, such block matrices would be associated to a decomposable representation, leading again to a contradiction.}
\end{proof}

We have established a link between ALFs and equivariant systems. Let us
notice that, if we have a dynamical system equivariant under a finite group,
such dynamical system can be seen as a \emph{cell-network} \cite{wrap181}.
A cell-network is a system of differential equations symmetric under a
group(oid) associated to a network that could have directed and multiple
    edges. So, there is a strict correspondence between cell-networks and the ALF
framework. For example, a $K_n$-ALF is a cell-network with complete graph
structure, $K_n$. This alternative point of view provides an extra tool in the
analysis of ALFs. Indeed, from Proposition \ref{prop:invariant_space}
it follows that $\fix(\textup{Aut}(\mathcal{G}))$ is flow-invariant,
and so, in turn, the consensus space is an invariant space for a $K_n$-ALF, or
in the language of cell-networks it is a pattern of synchrony, which is already
known from Corollary \ref{cor:consensus}. Actually, Corollary
\ref{cor:consensus} not only holds for all ALFs, but also tells us that the
consensus space consists of equilibria of the system, and is not
just an invariant set. However, Proposition \ref{prop:con_sym} does not
rely on the algebraic structure of the Laplacian, it is a general result for a
class of equivariant systems. Let us notice that the results following from
equivariance are, in general, more robust than the algebraic ones. Therefore,
the symmetry perspective could be very useful when some transformation is
applied to the system.

Let us notice that, when the order of $\textup{Aut}(\mathcal{G})$ is less than
$n$, we have $\dim \fix( \textup{Aut}(\mathcal{G}))>1$, so the fixed-point space
is related to clustering, and $C \subset \fix( \textup{Aut}(\mathcal{G}))$; we
illustrate such a case in the following example. We consider the path graph $P_n$ with $n> 2$, which has $\mathbb{Z}_2$ as
group of symmetry and $\textup{Ord}(\mathbb{Z}_2) = 2$. We can see that the equalities imposed by $\mathbb{Z}_2 \bm{x} = \bm{x}$ are
\begin{align*}
    x_1     & = x_n     \\
    x_2     & = x_{n-1} \\
    \vdots  &           \\
    x_{n-1} & = x_2     \\
    x_n     & =x_1 ,
\end{align*}
which do not lead to consensus in general, but to clustering $\{ \bm{x} \in \mathbb{R}^n \ | \ x_1 = x_n, \dots, x_n=x_1 \}$.

In summary, in this section we outlined a different perspective on the
consensus space. Let us recall that a weighted ALF, near the consensus
space, is equivalent to an ALF with simple graph structure (Proposition
\ref{prop:equivalence}). Thus, the weighted ALF acquires the property of
equivariance under the automorphism group of the (simple) graph; the symmetry
group is a subgroup of the symmetric group. Therefore, an ALF possesses also a
cell-network structure. At this point, forgetting momentarily the specific
properties of ALFs, we can study the connections between the consensus space
and the symmetry properties. It turns out that the consensus space arises as an
invariant space for a class of subgroups of $\mathfrak{S}_n$ (Proposition
\ref{prop:con_sym}).

We conclude this section by recalling a theorem highlighting the fundamental role
of the graph structure in the definition of a dynamic network. Although, as we
have seen, graphs and groups are intimately related, if we consider the logic
flow of definitions for dynamic networks the graph structure should be defined
first.

\begin{Theorem}[Frucht \cite{CM_1939__6__239_0}]\label{thm:frucht}
    For any finite group $\Gamma$, there exists a finite graph $\mathcal{G}$ such that $\textup{Aut}(\mathcal{G}) 	\cong \Gamma$.
\end{Theorem}

Theorem \ref{thm:frucht} ensures that for any given finite group we can
construct a graph which is invariant under the group. However, uniqueness is
not provided, au contraire, in general there are infinitely many graphs with
such a property \cite{biggs1993algebraic}. For this reason, the definition of
the graph structure of a dynamic network precedes the group properties.


\section{Perturbations}\label{sec:perturbations}

A perturbation is a small `change' of the differential equations governing the
system. The order of magnitude can be quantified by a parameter $\epsilon$, $0<
    \epsilon \ll 1$. A perturbation is called \emph{singular} if the solutions of
the perturbed ODEs have a different qualitative behavior with respect to the
solutions of the unperturbed system \cite{wechselberger2020geometric}. On the
other hand, \emph{regular} perturbations are the ones that preserve the
qualitative structure. For this reason, singular perturbations are the most
interesting if we aim to characterize a system qualitatively. In Appendix
\ref{app:singular}, we briefly review some of the standard terminology of
singular perturbations.


\subsection{Perturbed ALFs}
As mentioned in the introduction, we now consider perturbations of ALFs
given by
\begin{equation}\label{eq:perturbed_LDN}
    \dot{\bm{x}} = - L F(\bm{x}) + \epsilon H(\bm{x},\Lambda) ,
\end{equation}
where $0< \epsilon \ll 1$, $H(\bm{x},\Lambda) = (h_1(\bm{x},\Lambda), \dots, h_n(\bm{x},\Lambda))^\intercal$, and $\Lambda$ is a set of parameters.
The consensus space $C$ is always a set of equilibrium points for
$\epsilon=0$ (Corollary \ref{cor:consensus}). Moreover, $C$
is one-dimensional and hence generic perturbations are singular
\cite{wechselberger2020geometric}. The system as written in
\eqref{eq:perturbed_LDN} is in non-standard form. In many cases, it is useful
to transform a singularly perturbed system into a standard form, one reason
among all is that a large amount of results in singular perturbation theory are
stated for standard forms.

    \begin{Proposition}
        Let $l \in \mathcal{V}$, and $\tilde{\bm{x}}=(\tilde{x}_1, \dots, \tilde{x}_{n-1})^\intercal$ be an $(n-1)$-dimensional vector. Let $\tilde{L}$ be the $(n-1)\times(n-1)$ matrix obtained from the Laplacian $L$ by removing the $l$-th column and row, i.e.,
        \begin{equation}
            \tilde{L} :=
            \begin{pmatrix}
                L_{11}     & \dots  & L_{1 (l-1)}     & L_{1 (l+1)}     & \dots  & L_{1n}     \\
                \vdots     & \ddots & \vdots          & \vdots          & \ddots & \vdots     \\
                L_{(l-1)1} & \dots  & L_{(l-1) (l-1)} & L_{(l-1) (l+1)} & \dots  & L_{(l-1)n} \\
                L_{(l+1)1} & \dots  & L_{(l+1) (l-1)} & L_{(l+1) (l+1)} & \dots  & L_{(l+1)n} \\
                \vdots     & \ddots & \vdots          & \vdots          & \ddots & \vdots     \\
                L_{n1}     & \dots  & L_{n (l-1)}     & L_{n (l+1)}     & \dots  & L_{nn}
            \end{pmatrix},
        \end{equation}
        and let $v^{(l)}$ be the $(n-1)$-dimensional vector defined as follows
        \begin{equation}
            v^{(l)}:=(L_{1l}, \dots, L_{(l-1) l}, L_{(l+1) l}, \dots, L_{nl})^\intercal .
        \end{equation}
        The coordinate transformation
        \begin{equation}\label{eq:reduction}
            \begin{aligned}
                x_1 & \mapsto \tilde{x}_1 ,                                                      \\
                    & \vdots                                                                     \\
                x_l & \mapsto  x_l\left(k, \tilde{\bm{x}} \right):= k -  \sum_{j} \tilde{x}_j  , \\
                    & \vdots                                                                     \\
                x_n & \mapsto \tilde{x}_{n-1} ,
            \end{aligned}
        \end{equation}
        where $k$ is the constant of motion \eqref{eq:constant}, transforms system \eqref{eq:perturbed_LDN} into the standard form
        \begin{equation}\label{eq:perturbed_LDN_standard}
            \begin{aligned}
                \dot{\tilde{x}}_i & = - \sum_{j } \tilde{L}_{ij} f(\tilde{x}_j) - v^{(l)}_i  f\left(  x_l\left(k, \tilde{\bm{x}} \right)  \right) +
                \epsilon g_i (\tilde{\bm{x}}, k, \Lambda)                                                                                           \\
                \dot k            & = \epsilon \sum_j g_j (\tilde{\bm{x}}, k, \Lambda) ,
            \end{aligned}
        \end{equation}
        where $g_i (\tilde{\bm{x}}, k, \Lambda) := h_i \left(\tilde{x}_1, \dots,  x_l \left(k, \tilde{\bm{x}}\right), \dots, \tilde{x}_n ,\Lambda \right)$.
    \end{Proposition}

\begin{proof}
        Substituting the coordinate change Eq \eqref{eq:reduction} into Eq \eqref{eq:perturbed_LDN} we obtain the set of equations for $\tilde{x}_i$. Then, the equation for $k$ is retrieved by evaluating the scalar product $\langle \mathbf{1}, \dot{\bm{x}} \rangle$.   
    \end{proof}

Let us notice that the perturbation transforms the constant of motion into a
slow variable of the system. Since the constant $k$ is related to the
arithmetic mean of the states, $\langle \bm{x} \rangle$, perturbing the system,
in general, means that the states' density is changing, e.g., the concentration
of the system is increasing/decreasing. In other words, the perturbation acts
as a drift/dissipation in the undergoing diffusive evolution.
Notice that if the perturbation $H$ is `perpendicular' to the consensus space,
$\langle \mathbf{1} , H(\bm{x},\Lambda) \rangle = 0$, then $\dot k = 0$, which
in turn means that the perturbation problem is regular.

As mentioned, we are interested in singular perturbations for two major reasons. First of all, in the setting of our problem where the unperturbed system has a set of equilibria of dimension greater than zero, generic perturbations are singular. The second reason is that singular perturbations are the ones that could lead to a different qualitative behavior
of the system. In the following section, we look at the near consensus
behavior under a singular perturbation, and we characterize the dynamics for a
selected class of systems.



\section{Consensus dynamics}\label{sec:consensus}

In this section, we further specialize our analysis to complete networks, $K_n$. We will show how this assumption allows a reduction of the equations, and a detailed description of the consensus behavior will be given.


\subsection{The layer problem}

Let us consider a perturbed $K_n$-ALF in the standard form
\eqref{eq:perturbed_LDN_standard}. The layer problem, obtained by setting
$\epsilon$ to zero, reads as
\begin{equation}\label{eq:red_LDN}
        \dot{\tilde{x}}_i = - \left[ \sum_{j} \tilde{L}_{ij} f(\tilde{x}_j) + v^{(l)}_i f\left( x_l \left(k, \tilde{\bm{x}} \right) \right) \right] .
\end{equation}
The system described by \eqref{eq:red_LDN} is $(n-1)$-dimensional with one
parameter, $k$.
It is worth noting that although the graph structure in \eqref{eq:red_LDN}
seems to be lost, some information is readily available. On the one hand,
the term $\sum_{j} \tilde{L}_{ij}f(\tilde{x}_j)$ corresponds to a
$K_{(n-1)}$-ALF, i.e., a complete graph with $(n-1)$ nodes. On the other hand,
the scalar term $ v^{(l)}_i f\left( x_l \left(k, \tilde{\bm{x}} \right)
        \right)$, can be interpreted as a higher-order interaction
\cite{bick2021higher} naturally arising from the reduction performed.

Furthermore, and for our specific purposes, there is a more relevant property
that survives the transformation: The symmetry. In fact, starting with an
$\mathfrak{S}_n$-equivariant unperturbed system, after the transformation into
the standard form we get a layer problem which is
$\mathfrak{S}_{n-1}$-equivariant. This is a first example of robustness of the
symmetries. Since we are interested in the consensus dynamics, we consider
equation \eqref{eq:red_LDN} on the invariant space given by
$\fix(\mathfrak{S}_{n-1})$. Let us notice that, in general, the invariant space
$\fix(\mathfrak{S}_{n-1})$ is not globally stable. For example, in a
neighborhood of an attracting branch of the consensus space it is stable.
Instead, in a neighborhood of a repelling part of the consensus space it is
unstable. This means that the restriction to $\fix(\mathfrak{S}_{n-1})$
provides a robust description of the system only in the basin of attraction of
a stable region of $\fix(\mathfrak{S}_{n-1})$; in the unstable regions the
restriction works only if the initial conditions lie exactly in the invariant
space and if the perturbation preserves $\fix(\mathfrak{S}_{n-1})$. Then,
restricting equation \eqref{eq:red_LDN} to $\fix(\mathfrak{S}_{n-1})
        = \{ \tilde{\bm{x}} \in \mathbb{R}^{n-1} \ | \ \tilde{x}_1= \dots
        =\tilde{x}_{n-1} \}$ we obtain
\begin{equation}\label{eq:one_dim}
    \dot x  = - \left( f(x) - f(k - (n-1) x) \right) ,
\end{equation}
where $x:=\tilde{x}_1= \dots =\tilde{x}_{n-1}$, and we used the properties of the Laplacian of a complete graph. So, in the invariant space the system reduces to a one dimensional system (with one parameter). The former consensus space, $C=\fix(\mathfrak{S}_n)$, gives the equilibria of \eqref{eq:one_dim}. This is a consequence of the fact that the layer problem is a coordinate transformation of the unperturbed non-standard form Eq \eqref{eq:LDN}, for which we know the consensus space is a set of equilibria.

\begin{Remark}
        The parametrized consensus states \eqref{eq:equilibrium} given by $x^*_k = k/n$ are equilibria of \eqref{eq:one_dim}.
    \end{Remark}

Since the equilibria of the layer equation constitute the critical manifold for
the slow-fast system, then the consensus space is a (part of the) critical
manifold of the system.

\begin{Proposition}\label{prop:conditions}
    The stability of the consensus set \eqref{eq:equilibrium} is determined by the derivative of the response function. In particular, we can decompose the consensus into three components,
    \begin{equation}
        \begin{aligned}
            C^a=\left\{\bm{x}^*_k \in C \ | \  \textup{d}_xf(x^*_k) >0 \right\}   ,     \\
            C^r=  \left\{\bm{x}^*_k \in C \ | \   \textup{d}_xf(x^*_k) <0 \right\}    , \\
            C^s= \left\{\bm{x}^*_k \in C \ | \  \textup{d}_xf(x^*_k) =0 \right\}     ,
        \end{aligned}
    \end{equation}
    where $C^a$ is attracting (stable), $C^r$ is repelling (unstable), and $C^s$ is singular.
\end{Proposition}

\begin{proof}
    By straightforward computations we obtain the Jacobian of the system evaluated at the consensus, $ - \textup{d}_xf\left(k/n\right) n$, which implies the proposition.
\end{proof}

For $0<\epsilon\ll1$, thanks to Theorem \ref{thm:fenichel} (Fenichel), Appendix \ref{app:singular},  compact {submanifolds} of $C^{a/r}$ perturb to $O(\epsilon)$-close locally invariant slow manifolds of the singularly perturbed system
\eqref{eq:perturbed_LDN_standard} (denoted correspondingly by $C^{a/r}_\epsilon$).
For the singularities $C^s$ we need further
analyses. We are going to consider the case of \emph{transcritical}
singularities which, as we will show, turn out to be generic due to the
geometric properties of the problem.


\subsection{The slow-fast dynamics around the consensus space}\label{sec:sf_dynamics}

We start by considering a symmetry condition on the perturbation. While restrictive, we impose it to guarantee the invariance of $\fix(\mathfrak{S}_{n-1})$ for the perturbed system \eqref{eq:perturbed_LDN_standard}.

\begin{Lemma}\label{lm:pert_assumption}
    Let $H(\bm{x},\Lambda)$ be a perturbation such that $h_i = h$, $\forall i \neq l$, and $h_l = \tilde{h}$. Then, $\fix(\mathfrak{S}_{n-1})$ is an invariant space for the perturbed system \eqref{eq:perturbed_LDN_standard}.
\end{Lemma}

    \begin{proof}
        Let us consider the equations in standard form \eqref{eq:perturbed_LDN_standard} with a perturbation satisfying the conditions $h_i = h$, $\forall i \neq l$, and $h_l =\tilde{h}$. So, if we restrict the equations to
        \begin{equation}\label{eq:xk_plane}
            \{ (x,k) \in \mathbb{R}^2 \ | \ x := \tilde{x}_1 = \dots = \tilde{x}_{n-1}  \}
        \end{equation}
        we obtain
        \begin{equation}\label{eq:plane_syst}
            \begin{aligned}
                \dot x & = - \left[ f(x) - f(k - (n-1) x) \right] + \epsilon g(x, k , \Lambda)       \\
                \dot k & = \epsilon  (n-1)  g(x, k , \Lambda) + \epsilon \tilde{g}(x, k , \Lambda) ,
            \end{aligned}
        \end{equation}
        where $g(x, k , \Lambda) := h(x, \dots, k - (n-1) x, \dots, x, \Lambda)$, and $\tilde{g}(x, k , \Lambda) := \tilde{h}(x, \dots, k - (n-1) x, \dots, x, \Lambda) $. Then $\fix(\mathfrak{S}_{n-1})$ is an invariant space.
    \end{proof}

Assuming a perturbation satisfying Lemma \ref{lm:pert_assumption}, we are allowed to continue our analysis on the invariant space \eqref{eq:xk_plane}, which we will refer to as $(x,k)$-plane.
We remark once again that, in a neighborhood of a stable region of
$\fix(\mathfrak{S}_{n-1})$, the reduction is robust; while in unstable regions
the reduction is preserved only by fine tuning of initial conditions and
perturbation.

\subsubsection{Genericity of transcritical intersections}\label{sec:plynomial_response_function}

The analysis to be performed in this section is local. Thus, let us now consider that the response function can be written as
\begin{equation}\label{eq:poly_resp}
    f(x) = \sum_{j=1}^N a_j x^j +\mathcal O(x^{N+1}),
\end{equation}
where $a_j$ are real constant coefficients, $N \in \mathbb{N}$, and we assume that at least one of the coefficients $a_j$, $j\neq1$, is different from zero. We know that the consensus space, for an ALF, does not depend on the particular choice of the response function. On the other hand, the remaining equilibrium sets, that together with $C$ give rise to the full critical manifold, depend on the form of the response function. By detailing the local form of $f$, we are able to further describe the {local shape of the other branches of the critical manifold} crossing the consensus space. We focus our analysis on a neighborhood of the consensus space.
The following proposition shows that if there is a singularity on the consensus space, implying that another branch of equilibria intersects with the consensus space, then we can derive the linear approximation of such intersecting branches of equilibria.

\begin{Proposition}\label{prop:tangents}
    Let the response function $f$  be of the form \eqref{eq:poly_resp} with $N\geq2$. In the $(x,k)$-plane, if a branch of the critical manifold intersects the consensus space, then such a branch has the local form
    \begin{equation}\label{eq:tangents}
        \kappa(x) = 2 x^s_k + (n-2)x ,
    \end{equation}
    where $(x_k^s,\ldots,x_k^s)\in C^s$.
\end{Proposition}

\begin{proof}
    Let us consider the function $  \tilde{\phi}(x,y) := - [f(x) - f(y + x)]$ and the auxiliary system
    \begin{equation}
        \begin{aligned}
            \dot{x} & = \tilde{\phi}(x,y) \\
            \dot y  & = 0 ,
        \end{aligned}
    \end{equation}
    where $y$ acts as a parameter, with $y$ close to zero. We are going to evaluate the slope of the tangents of the curves of zeros of $\tilde{\phi}(x,y)$ intersecting the line $y=0$. Considering
    the response function \eqref{eq:poly_resp}, we expand $f(x + y)$ by using the binomial theorem,
    \begin{align*}
        f(y+x) & = \sum_{j=1}^N a_j (y+x)^j  +\mathcal O((x+y)^{N+1})                                                        \\
               & = \sum_{j=1}^N a_j x^j + \sum_{j=1}^N  \sum_{k=1}^j  \binom{j}{k} a_j y^k x^{j-k} +\mathcal O(x^{N+1})+y\mathcal O(y^ux^v),
    \end{align*}
    where, for brevity, we collect in $\mathcal O(y^ux^v)$ all higher order terms of such order for all combinations $u+v=N$.
    So, we can rewrite $\tilde{\phi}(x,y)$ as
    \begin{equation}
        \tilde{\phi}(x,y) = y\left( \sum_{j=1}^N  \sum_{l=0}^{j-1}  \binom{j}{l+1} a_j y^l x^{j-l-1}+\mathcal O(y^ux^v)\right) ,
    \end{equation}
    where we factored out $y$, and set $l:=k-1$. We regularize $\tilde{\phi}(x,y)$ by removing the line $y=0$ from the set of zeros, i.e. $\tilde{\phi}(x,y)=:y \tilde{\phi}^\textup{reg}(x,y)$, where
    \begin{equation}
        \tilde{\phi}^\textup{reg}(x,y) = \sum_{j=1}^N  \sum_{l=0}^{j-1}  \binom{j}{l+1} a_j y^l x^{j-l-1}+\mathcal O(y^ux^v) .
    \end{equation}
    We continue by disregarding the higher-order terms as they do not play a role in our forthcoming arguments. The regularized auxiliary system reads
    \begin{equation}
        \begin{aligned}
            \dot{x} & = \tilde{\phi}^\textup{reg}(x,y) \\
            \dot y  & = 0 .
        \end{aligned}
    \end{equation}
    Let us compute the Jacobian, $J$, for the regularized auxiliary system. We get
    \begin{equation}
        J = \begin{pmatrix}
            \frac{\partial \tilde{\phi}^\textup{reg} }{\partial x} & \frac{\partial \tilde{\phi}^\textup{reg} }{\partial y} \\ 0 & 0
        \end{pmatrix} ,
    \end{equation}
    where
    \begin{align}
        \frac{\partial \tilde{\phi}^\textup{reg} }{\partial x} & = \sum_{j=1}^N \sum_{l=0}^{j-1} \binom{j}{l+1} a_j (j-l-1) y^l
        x^{j-l-2} ,                                                                                                             \\ \frac{\partial \tilde{\phi}^\textup{reg} }{\partial y} & = \sum_{j=1}^N \sum_{l=0}^{j-1} \binom{j}{l+1} a_j l y^{l-1} x^{j-l-1}
           .
    \end{align}
    Since we are interested in what happens at the intersection with $y=0$, we evaluate the Jacobian on that line, namely
    \begin{align}
        J_0:= J\big|_{y=0} & = \sum_{j=1}^N j(j-1) a_j x^{j-2}
        \begin{pmatrix}
            1 & \frac{1}{2} \\
            0 & 0
        \end{pmatrix} .
    \end{align}
    The eigenvalues of $J_0$ are determined by the equation $\det(J_0 - \mu \id) = 0$, which in our case reads
    \begin{equation}
        - \mu \left[ \sum_{j=1}^N j(j-1) a_j x^{j-2} - \mu \right] =0 ,
    \end{equation}
    and so the eigenvalues are
    \begin{align}
        \mu_1 & = 0  ,                              \\
        \mu_2 & = \sum_{j=1}^N j(j-1) a_j x^{j-2} .
    \end{align}
    Let us notice that, in general, for the regularized auxiliary system the Jacobian $J_0$ cannot be interpreted as the linearisation of the system, because $y=0$ is no more an equilibrium. However, there could be other equilibria on the line $y=0$ at the intersections with the other curves of equilibria. At such points the eigenvector with zero eigenvalue gives the direction of the tangent of the intersecting curve of zeros. The eigenvector $\tilde{v}^{(\mu_1)}=(\tilde{v}^{(\mu_1)}_x , \tilde{v}^{(\mu_1)}_y)^\intercal$ must satisfy the equation $J_0\tilde{v}^{(\mu_1)} = 0$, that is
    \begin{equation*}
        \sum_{j=1}^N j(j-1) a_j x^{j-2} \left(\tilde{v}^{(\mu_1)}_x + \frac{1}{2} \tilde{v}^{(\mu_1)}_y\right) = 0.
    \end{equation*}
    Therefore, since for $j\neq1$ at least one $a_j\neq0$,  $\tilde{v}^{(\mu_1)}_x = - \frac{1}{2} \tilde{v}^{(\mu_1)}_y$, and so $\tilde{v}^{(\mu_1)} \propto (1 , -2)^\intercal$, where $\propto$ is the proportionality symbol. Notice that $\tilde{v}^{(\mu_1)}$ does not depend on $x$, which means that wherever the intersection takes place, the tangent will have the direction of $\tilde{v}^{(\mu_1)}$. At this point, we need to obtain the equivalent vector for the system of our interest \eqref{eq:one_dim}. Note that the equation $y=k-nx$ puts in relation the auxiliary system with the layer problem \eqref{eq:one_dim}. Thanks to this relation we can construct the transformation connecting the tangent vectors of the two systems, in particular we have
    \begin{align*}
        \Pi & := \begin{pmatrix}
                     \frac{\partial x(x,k) }{\partial x} & \frac{\partial x(x,k) }{\partial k} \\ \frac{\partial y(x,k) }{\partial x} & \frac{\partial y(x,k) }{\partial k}
                 \end{pmatrix}
        = \begin{pmatrix}
              1  & 0 \\
              -n & 1
          \end{pmatrix} ,
    \end{align*}
    and then the inverse transformation reads
    \begin{equation}
        \Pi^{-1} =
        \begin{pmatrix}
            1 & 0 \\
            n & 1
        \end{pmatrix} .
    \end{equation}
    Thus, the tangents of the critical manifold at the intersections with the consensus space, in the original system, have direction given by $v_{\mu_1} = \Pi^{-1} \tilde{v}_{\mu_1} \propto (1 , n-2)$, which proves that the slope of the tangent in the $(x,k)$-plane is $n-2$. Therefore, the tangent lines have an equation of the form $\kappa(x) = \tilde{\kappa} + (n-2)x $. In order to find $\tilde{\kappa}$, we use the fact that $\kappa(x^s_k)=n x^s_k$, obtaining $\tilde{\kappa}=2 x^s_k$; the statement is proven. A sketch of the result of this proposition is shown in Figure \ref{fig:cascade}.
\end{proof}

\begin{figure}[htbp]
    \centering
    \begin{tikzpicture}
        \node at (0,0){\includegraphics[scale=0.5]{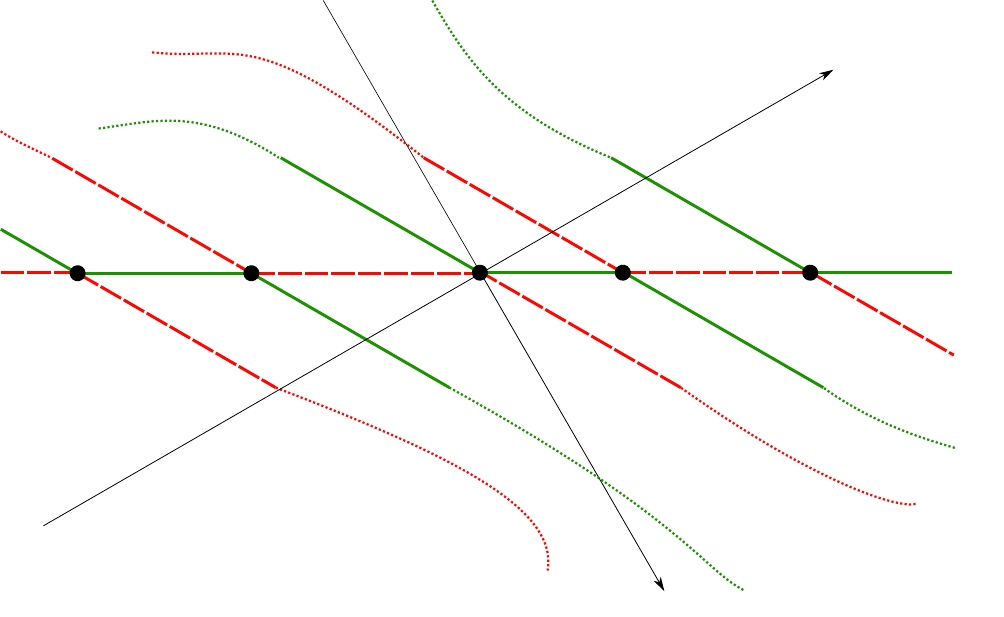}};
        \node at (3,2) {$k$};
        \node at (4.1,0.25) {$C$};
        \node at (1.5,-2.5) {$x$};
    \end{tikzpicture}
    \caption[Local structure of the critical manifold near the consensus space.]{
        Local structure of the critical manifold near the consensus space in the $(x,k)$-plane. In the figure, the $(x,k)$-plane is rigidly rotated in order to visualize the consensus space as a horizontal line. The black dots represent singular points, the colored lines crossing the consensus space at the singular points are other branches of the critical manifold. We draw in dashed red the unstable sections of the critical manifold, while the stable ones are drawn in solid green. We emphasize that, in general, the branches of equilibria that cross the consensus manifold are not straight lines, but due to Proposition \ref{prop:tangents}, the tangents at the singular points are all parallel to each other. }
    \label{fig:cascade}
\end{figure}

Proposition \ref{prop:tangents} tells us that, up to linear approximation, all
other branches of equilibria of the critical manifold crossing the consensus
space are parallel to each other and are in fact transverse to the consensus
space. Moreover, for $n>2$ such branches of the critical manifold locally are also transverse to the fast-flow. As a consequence, we have that
pitchfork singularities are not possible on the consensus space for the class
of ALFs under examination, this statement is also confirmed by checking the
conditions for pitchfork singularities \cite{Krupa_2001}. By a similar
reasoning, fold singularities on the consensus space are also excluded.
Finally, Hopf singularities do not appear either by the fact that the layer
system has always real eigenvalues. Thus, one can conclude that for $K_n$-ALFs
transcritical singularities are generic on the consensus space, in the sense of
codimension one bifurcations; while on other branches of the critical manifold
it is possible to have other singularities, see some examples in Section
\ref{sec:ex_and_sim}.

\begin{Remark}
    Although in this section we have only considered complete networks, we conjecture that other topologies also present, generically, transcritical singularities along the consensus space. The main argument is that, no matter the network structure, as long as the consensus space exists, it is a \emph{one dimensional linear subspace} of the phase-space. Singularities of this linear subspace appear as intersections with other branches of the critical manifold induced by the nonlinear response function $F$. Locally, and generically, one would expect that these intersections are transversal, leading to transcritical singularities (possibly in higher dimensions). In particular, recalling the result of Proposition \ref{prop:invariance}, we know that the symmetries induced by the invariance of the response function generates linear branches of the critical manifold intersecting transversely the consensus set. The great advantage of considering complete networks is the reduction to a planar problem. This reduction requires extra work for other network topologies. Nevertheless, our conjecture will be supported by numerical simulations in Section \ref{sec:ex_and_sim}.
\end{Remark}

\subsubsection{Nondegenerate transcritical singularities}
We now state the conditions for (nondegenerate) transcritical singularities.

\begin{Proposition}\label{prop:trans}    
        Let $\bm{x}^s_k \in C^s$ be a singular consensus point. Let $f$ and $h_i$, $i=1, \dots, n$, be $\mathcal{C}^m$-functions with $m \geq 3$, and assume that Lemma \ref{lm:pert_assumption} holds. If
        \begin{align}
            \textup{d}_x^2  f(x^s_k)                             & \neq 0 , \\
            n                                                    & \neq 2 , \\
            \langle \mathbf{1} , H(\bm{x}^s_k , \Lambda) \rangle & \neq 0 ,
        \end{align}
        then $\bm{x}^s_k$ is a nondegenerate transcritical singularity of \eqref{eq:plane_syst}.   
\end{Proposition}

\begin{proof}
    We check the conditions for transcritical singularities on planar systems provided by Krupa and Szmolyan \cite{Krupa_2001}. Given a planar system of the form \eqref{eq:standard_fast}
    then the origin is a transcritical singularity if
    \begin{align}
        f(0,0,0)                                             & = 0 , \label{al:a}    \\
        \partial_x f(0,0,0)                                  & = 0 , \label{al:b}    \\
        \partial_y f(0,0,0)                                  & = 0 , \label{al:c}    \\
        \begin{vmatrix}
            \partial^2_x f(0,0,0)    & \partial^2_{xy} f(0,0,0) \\
            \partial^2_{yx} f(0,0,0) & \partial^2_{y} f(0,0,0)
        \end{vmatrix} & < 0 , \label{al:d}                          \\
        \partial^2_x f(0,0,0)                                & \neq 0 , \label{al:e} \\ g_0:=g(0,0,0) & \neq 0 .
           \label{al:f}
    \end{align}
    We now prove that these conditions applied to our system lead to the conditions of the statement. Conditions Eqs \eqref{al:a}--\eqref{al:c} are satisfied because $x^s_k$ is a singular consensus point. For the non-degeneracy condition \eqref{al:d}, we obtain
    \begin{equation}
        \textup{d}_x^2  f(x^s_k) \neq 0 .
    \end{equation}
    The transversality condition \eqref{al:e} gives,
    \begin{equation}
        n (n-2)  \textup{d}_x^2  f(x^s_k) \neq 0 ,
    \end{equation}
    which implies $n \neq 2$. Finally, the last condition \eqref{al:f} leads to $\langle \mathbf{1} , H(\bm{x}^s_k , \Lambda) \rangle \neq 0$.
\end{proof}

We consider the case where all the singularities on the consensus space are
nondegenerate transcritical, i.e., the conditions of Proposition
\ref{prop:trans} are met on $C^s$. So, in our current setting, the consensus
space consists of a cascade of transcritical singularities changing the
stability from attracting to repelling, and vice-versa. If at $x^s_k$ the
stability transition, in the direction of the slow flow, is from
attracting to repelling, we call $x^s_k$ a \emph{type-1 transcritical point}.
Type-1 transcritical points are characterized by the sign ratio
\begin{equation}
    \rho :=  \frac{\sgn\left( \textup{d}_x^2 f(x^s_k)\right)}{\sgn \left( \langle \mathbf{1} , H(\bm{x}^s_k , \Lambda) \rangle \right)} = -1 .
\end{equation}
If at $x^s_k$ the stability transition, in the direction of the slow flow, is from repelling to attracting, we call $x^s_k$ a \emph{type-2 transcritical point}. Type-2 transcritical points are characterized by the sign ratio $\rho = 1$.

Under perturbations the consensus space, being a component of the critical manifold, extends to a slow manifold as predicted by Fenichel, Theorem \ref{thm:fenichel}. More specifically, compact normally hyperbolic submanifolds of the consensus perturb to slow manifolds with the same stability properties. The continuation of such slow manifolds close to a singularity depends on the perturbation. Suppose we follow a trajectory starting close to an attracting section of the consensus space. When such trajectory reaches a neighborhood of a singularity its behavior depends specifically on the perturbation and on the geometry of the aforementioned slow manifolds. In particular, a trajectory reaching a neighborhood of a transcritical point could proceed in three different manners.
\begin{description}
    \item[{Exchange of stability}] {The trajectory exchanges to another branch, which intersects the consensus space transversely, of the critical manifold. In other words, the trajectory does not follow anymore the consensus space, but another stable branch of the critical manifold that results from the nonlinear response function.}

    \item[{Fast escape}] {The trajectory, once traversing a small neighborhood of the singularity, continues along the fast flow.}
    \item[{Bifurcation delay}] {The trajectory continues along an \emph{unstable} slow manifold that is perturbed from a normally hyperbolic repelling submanifold of the consensus space for a $O(1)$ time.}
\end{description}

The third characteristic behavior is called \emph{canard}, and acts as a separating case between the other two {\cite{kuehn2015multiple, wechselberger2020geometric}}.
Usually, the presence of a canard in a planar system is associated to a critical value of the parameters of the system. The description in terms of the parameters of the system has the advantage of enabling the unfolding of the singularity to be studied. Here, however, we adopt a geometric perspective on the problem. Near fold singularities a geometric perspective on canard solutions in the plane has been treated extensively in \cite{de2021canard}. In our case of study we show that the requirement for a canard solution can be stated as an algebraic condition on the perturbation. 

\begin{Proposition}\label{prop:canard}
    Let $\bm{x}^s_k \in C^s$ be a singular consensus point satisfying the conditions of Proposition \ref{prop:trans}. If the perturbation satisfies the property
    \begin{equation}
        H(\bm{x}^s_k,\Lambda) \propto \mathbf{1},
    \end{equation}
    where $\propto$ is the proportionality sign, then for $\epsilon$ sufficiently small, system \eqref{eq:perturbed_LDN_standard} admits canard solutions.
\end{Proposition}
When a perturbation satisfies the condition of Proposition \ref{prop:canard}, we refer to it as \emph{critical} perturbation, and we use the notation $H^\textup{crit}(\bm{x},\Lambda)$.

\begin{proof}
    We use again the results from \cite{Krupa_2001} to check that our statement holds. Considering a system of the form \eqref{eq:standard_fast}, satisfying conditions \eqref{al:a}--\eqref{al:f}, then there exists a parameter $\lambda$ controlling the behavior at the transcritical singularity (which we remind to be at the origin for this benchmark system). The parameter $\lambda$ is given by
    \begin{equation}
        \lambda = \frac{1}{|g_0| \sqrt{\beta^2 - \gamma \alpha }} (\delta \alpha + g_0 \beta) ,
    \end{equation}
    where
    \begin{align*}
        \alpha = \frac{1}{2} \partial^2_x f(0,0,0) , & \qquad \beta = \frac{1}{2} \partial^2_{xy} f(0,0,0) , \\ \gamma = \frac{1}{2} \partial^2_y f(0,0,0) , &\qquad \delta = \partial_\epsilon f(0,0,0) .
    \end{align*}
    So, the parameter $\lambda$ for the system \eqref{eq:plane_syst}, at a transcritical point $x_k^s$, reads
    \begin{equation}
        \lambda = - \rho \frac{h(\bm{x}^s_k, \Lambda) + (n-1) \tilde{h}(\bm{x}^s_k, \Lambda) }{ \tilde{h}(\bm{x}^s_k, \Lambda) + (n-1) h(\bm{x}^s_k, \Lambda) } .
    \end{equation}
    For type-1 transcritical points, we have $\rho =-1$, and the condition for canards is $\lambda = \lambda^\textup{crit} :=1$. So, we obtain the equation
    \begin{equation}
        (n-2)\tilde{h}(\bm{x}^s_k, \Lambda) = (n-2) h(\bm{x}^s_k, \Lambda) .
    \end{equation}
    Since the case $n=2$ is excluded by Proposition \ref{prop:trans}, we have $\tilde{h}(\bm{x}^s_k, \Lambda) = h(\bm{x}^s_k, \Lambda)$. Therefore the critical perturbation $H^\textup{crit}(\bm{x},\Lambda)$ at $\bm{x}^s_k$ reads
    \begin{align*}
        H^\textup{crit}(\bm{x}^s_k,\Lambda) & = (h(\bm{x}^s_k,\Lambda), \dots, h(\bm{x}^s_k,\Lambda), \dots, h(\bm{x}^s_k,\Lambda))^\intercal \\
                                            & =h(\bm{x}^s_k,\Lambda)  \mathbf{1},
    \end{align*}
    which agrees with the statement. In order to complete the proof, we show that for a type-2 transcritical point a critical perturbation admits faux-canards. Let us recall that faux-canards for planar systems are generic, and appear for $\lambda < \lambda^\textup{crit}$ \cite{Krupa_2001}. Using the fact that for a type-2 transcritical point $\rho = 1$, we get $\lambda=-1$. 
\end{proof}

In Figure \ref{fig:blowup} we sketch the blow-up of the transcritical
points under a critical perturbation
\cite{Krupa_2001,10.1007/s10440-014-9994-9,https://doi.org/10.48550/arxiv.1901.01402}.
The blow-up picture makes clear the different behavior of the type-$1/2$
transcritical points and illustrates how the continuation of the consensus
space occurs.

\begin{figure}[htbp]
    \centering
    \def\svgwidth{1\columnwidth}
    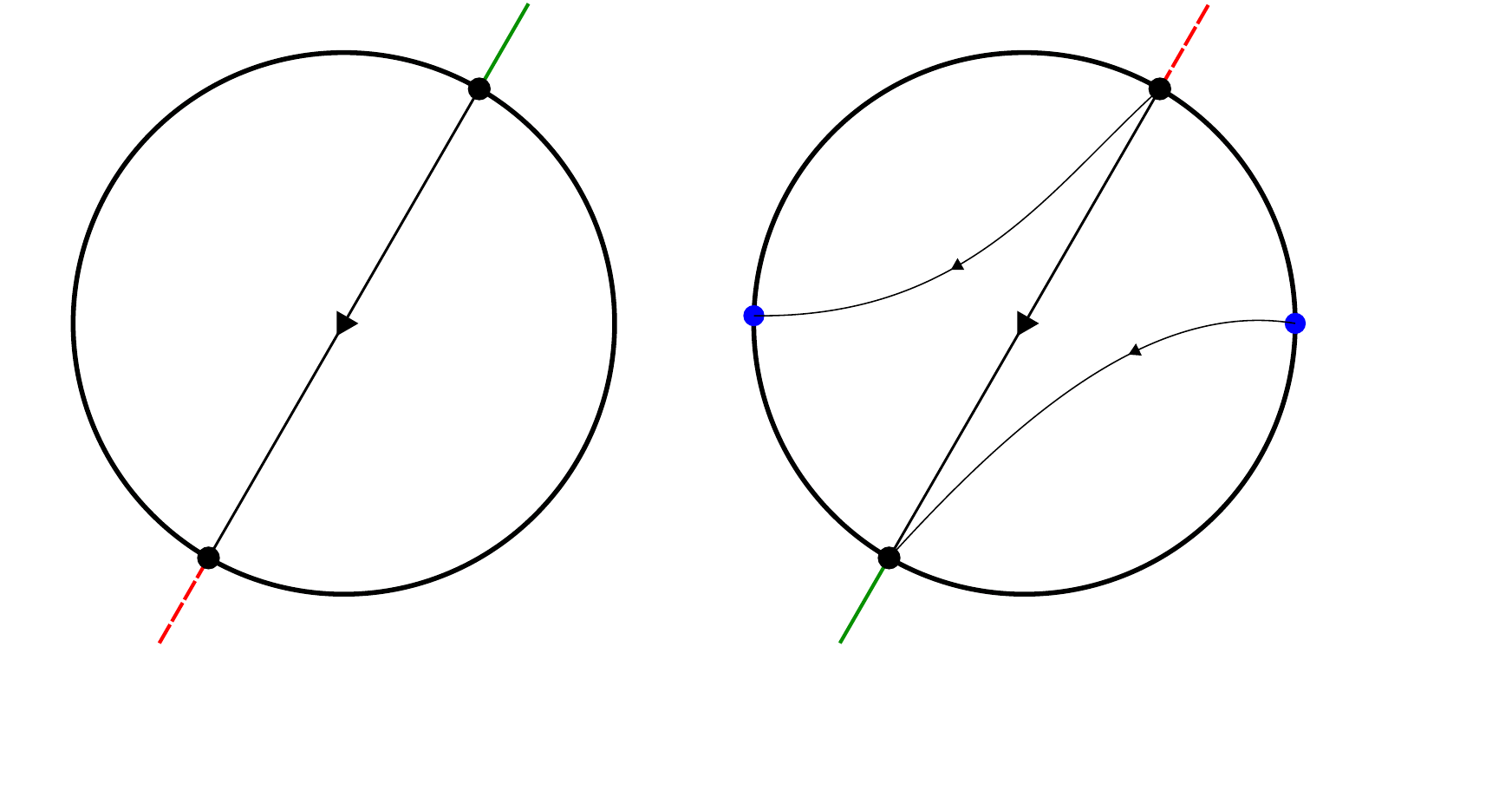

    \caption{Sketch of blown-up transcritical points on the consensus space. The picture displays typical blow-up phase portraits for a type-$1$ transcritical point (left), and a type-$2$ transcritical point (right); in both cases, we are considering a critical perturbation. For a detailed exposition on the blow-up phase portraits for transcritical singularities see \cite{Krupa_2001,10.1007/s10440-014-9994-9}. The points $\bar{x}^{s,-/+}_{a/r}$ correspond to the singularity on the consensus space, while the points $\bar{k}^{-/+}_{a/r}$ correspond to the singularity on the branch of the critical manifold crossing the consensus space. The labels $-/+$ indicate if we are below, or above, the transcritical point, and the labels $a/r$ stand for attracting, and repelling, respectively. The attracting curves of equilibria are shown in solid green, while the repelling ones in dashed red. The blue points are entering or exit points depending on the direction of the curves drawn.}
    \label{fig:blowup}
\end{figure}

Note that we can interpret the result of Proposition \ref{prop:canard} as follows: \emph{If the perturbation is tangent to the consensus space at a singularity, then canard solutions exist}. As a matter of fact, assuming a perturbation that is (everywhere) tangent to the consensus space, it is possible to prove a result that holds for \emph{every} network structure.

    \begin{Proposition}\label{prop:canard_generic}
        Given an ALF with arbitrary graph structure $\mathcal{G}$, and  $H(\bm{x},\Lambda)$ a perturbation satisfying $H( \bm{x}^*,\Lambda)  \propto \mathbf{1}$ and $H( \bm{x}^*, \Lambda) \neq 0$, $\forall \bm{x}^* \in C$. Then the consensus space is a trajectory of the perturbed system.
    \end{Proposition}
    \begin{proof}
        From Corollary \ref{cor:consensus} we know that the consensus space is an equilibrium of the unperturbed system; moreover, since $H(\bm{x}^*,\Lambda)  \propto \mathbf{1}$, $\forall \bm{x}^* \in C$, we have that $C$ is an invariant space for the perturbed vector field. So, projecting equation  \eqref{eq:perturbed_LDN} onto $C$ we obtain
        \begin{equation}
            \dot{\bm{x}}^* = \epsilon H(\bm{x}^*,\Lambda) ,
        \end{equation}
        $ \bm{x}^* \in C$. Given the assumption that the perturbation is non-zero on the consensus space, which is a one-dimensional space, we obtain that $C$ is a trajectory of the perturbed system.
    \end{proof}

In essence, Proposition \ref{prop:canard_generic} gives some conditions under which a perturbed equivariant ALF \eqref{eq:perturbed_LDN} exhibits a maximal canard. Let us also notice that Proposition \ref{prop:canard_generic} holds for all simple graphs, and it does not make any particular assumption on the singularities of the consensus space. In other words, paying the price of making more stringent assumptions on the perturbation we can cover the cases where the slow-fast theory would be difficult, or not even possible to apply.


\section{Numerical simulations}\label{sec:ex_and_sim}

In this section, we present several examples of ALFs, exploiting generic and
canard behavior from different perspectives. The simulations we show are performed in Mathematica \cite{math}, where the precision of the numerical simulations can be easily set. We show phase portraits, time
series and spatio-temporal plots highlighting the nonlinear diffusion and drift
mechanism on non-complete graphs. In view of Proposition
\ref{prop:equivalence}, our simulations are restricted to simple graphs but we
emphasize that for weighted graphs one would obtain equivalent results.

\subsection{$K_3$ networks: Generic and canard behavior}

We start by looking at a $K_3$-ALF considering two nonlinear response
functions, one that is even $f(x)=(x-1)^2(x+1)^2$ (Figure
\ref{fig:K3_even}), and one odd $f(x)=(x-1)(x+1)^2$ (Figure
\ref{fig:K3_odd}). The two response functions induce different critical
manifolds, Figures \ref{fig:generic_3nodes_pp} and
\ref{fig:generic_3nodes_f2_pp}. Nevertheless, the typical behavior is similar,
see Figures \ref{fig:generic_3nodes_ts} and
\ref{fig:generic_3nodes_f2_ts}.

It is worth recalling now a few issues with numerical simulations of
transcritical points, see more details in \cite{engel2020extended}. Roughly speaking, canard solutions, as shown in Figure  \ref{fig:reduced_canards}, contract first towards the critical manifold, cross a singularity, and then follow the unstable branch of the critical manifold enough time to compensate the aforementioned contraction before leaving an $\mathcal O(\epsilon)$-neighborhood of the critical manifold. For the particular case of a $K_3$-ALF with response function $f(x)=(x-1)^2(x+1)^2$, one can check that the transverse eigenvalue along the critical manifold is $-\frac{4}{9}k(9-k^2)$ implying that the contraction and expansion rate is symmetric with respect to the origin (this balance is more
precisely given by the entry-exit
relation \cite{engel2020extended,kuehn2015multiple}). Such a time symmetry
should be visible as a spatial symmetry in correspondence with a transcritical
point on the consensus manifold. As one can see in Figure
\ref{fig:reduced_canards}, the symmetry described is not very well reproduced
by the numerical simulation, although the computations are performed with a
working precision of $50$ digits. There are several factors that induce such a
problem. First, we have that the symmetry is exact only in the limit for
$\epsilon \to 0$, so we expect that a numerical simulation will just
approximate such symmetry. The simulation displayed in Figure
\ref{fig:reduced_canards} is performed with $\epsilon=1/10$, which is quite
large to clearly see this effect. However, by lowering $\epsilon$ one would
need to further increase the working precision, leading to problems with the
time performance of the CPU. Since our present work has a theoretical focus, we
do not dive deeper in the numerical analysis of these delicate systems, leaving
the problem of highly accurate simulations for future works.

\begin{figure}[htbp]
	\centering
	\begin{subfigure}[T]{.35\textwidth}
		\centering
		\includegraphics[width=\linewidth]{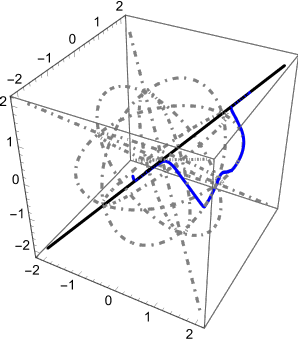}
		\caption{}
		\label{fig:generic_3nodes_pp}
	\end{subfigure}
	\begin{subfigure}[T]{.35\textwidth}
		\vskip 5em
		\includegraphics[width=\linewidth]{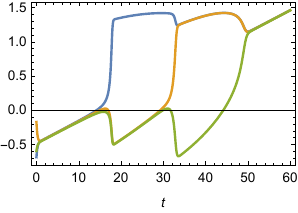}
		\caption{}
		\label{fig:generic_3nodes_ts}
	\end{subfigure}
	\begin{subfigure}[T]{.35\textwidth}
		\centering
		\includegraphics[width=\linewidth]{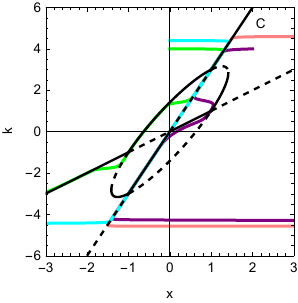}
		\caption{}
		\label{fig:reduced_canards}
	\end{subfigure}
	\begin{subfigure}[T]{.35\textwidth}
		\includegraphics[width=\linewidth]{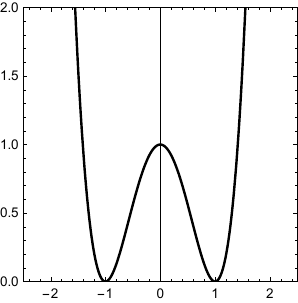}
		\caption{}
		\label{fig:f_resp_one}
	\end{subfigure}
	\caption{$K_3$-ALF with response function $f(x)=(x-1)^2(x+1)^2$ (see Figure \ref{fig:f_resp_one}), and $\epsilon=1/10$. Figures \ref{fig:generic_3nodes_pp} (phase portrait) and \ref{fig:generic_3nodes_ts} (time series) represent the same simulation of the system under a positive random constant perturbation, $ H = (0.484215,0.332579,0.155191)^\intercal$, with random initial conditions in the interval $[-1,0]$ for each variable, $\bm{x}(0)=(-0.686594, -0.171329, -0.595644)^\intercal$. As a consequence of the choice of the perturbation the slow dynamics is pointing upward. In the phase portrait \ref{fig:generic_3nodes_pp} consensus is the solid black line, while the other branches of the critical manifold are the dash-dotted gray curves; the trajectory of the system is drawn in blue.
		Figure \ref{fig:reduced_canards} displays the reduced system and four canard
		solutions, with initial conditions $(x_0,k_0)$ respectively $(2,4)$ (purple),
		$(0,4)$ (green), $(3,46/10)$ (pink), $(0,44/10)$ (cyan). Such solutions are
		obtained with a working precision of $50$ digits. The critical
		perturbation, $H^\textup{crit}=-\mathbf{1}$, is constant and negative, and,
		therefor, the slow dynamics is driving the system to lower values. In
		Figure \ref{fig:reduced_canards} the attracting section of the critical
		manifold, $C$, are represented with solid lines, while the repelling
		ones are dashed.
	}
	\label{fig:K3_even}
\end{figure}
\begin{figure}[htbp]
	\centering
	\begin{subfigure}[T]{.4\textwidth}
		\vskip 4em
		\centering
		\includegraphics[width=\linewidth]{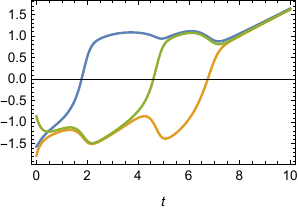}
		\caption{}
		\label{fig:generic_3nodes_f2_ts}
	\end{subfigure}
	\begin{subfigure}[T]{.4\textwidth}
		\includegraphics[width=\linewidth]{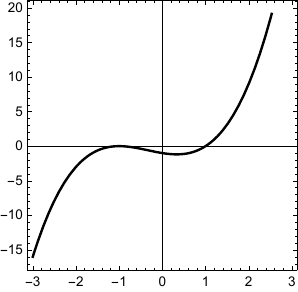}
		\caption{}
		\label{fig:f2_resp}
	\end{subfigure}
	\vskip 1em
	\begin{subfigure}{.4\textwidth}
		\centering
		\includegraphics[width=\linewidth]{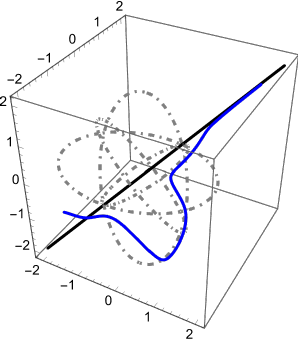}
		\caption{}
		\label{fig:generic_3nodes_f2_pp}
	\end{subfigure}
	\caption{$K_3$-ALF with response function $f(x)=(x-1)(x+1)^2$ (see Figure \ref{fig:f2_resp}), and $\epsilon=1/10$. Figures \ref{fig:generic_3nodes_f2_pp} (phase portrait) and \ref{fig:generic_3nodes_f2_ts} (time series) represent the same simulation of the system under a positive random constant perturbation, $ H=(0.723401,0.0273475,0.156751)^\intercal$, with random initial conditions in the interval $[-2,0]$, $\bm{x}(0)=(-1.56706,-1.77609,-0.866972)^\intercal$. The perturbation induces a drift of the system to higher values, i.e., the slow dynamics is pointing upward. Each of the \enquote{jumps} observed in the time-series correspond to a passage through a singularity (intersection of branches of the critical manifold) in the phase-portrait.
	}
	\label{fig:K3_odd}
\end{figure}
There is another detail to take into account: In Figure
\ref{fig:reduced_canards} we have canard solutions that cross three
transcritical points on the consensus space without leaving the repelling
intermediate section (light blue and pink curves). Here the question is: Where
should we expect the canards to leave the consensus manifold? As we just
discussed, we expect some symmetry, but here we need to be more careful stating
with respect to which transcritical point. We argue that, for these particular
solutions, the intermediate section has a null effect, and, therefore, the
solution should leave the consensus after the third transcritical point
(counting from top to bottom) after a distance (approximately) equal to
the distance traveled close to the attracting section before the first
transcritical point. In order to prove the aforementioned statement, we use the
slow-divergence integral 
 \cite{dumortier2011slow,de2015slow,de2021canard}. So,
let us consider the slow-divergence integral along the consensus manifold for
the system under analysis. Since the perturbation is constant the divergence of
the vector field in the $(x,k)$-plane, evaluated on the consensus space, is
simply $- n \textup{d}_x f (k/n)$. Therefore the slow-divergence integral
becomes
\begin{equation}
    -n \int \textup{d}_x f (k/n) dk.
\end{equation}
We now notice two facts. First, the intermediate section is divided into two spatially equal sections by a transcritical point. Second, on the first subsection $\textup{d}_x f (k/n)>0$, and on the second subsection we have $\textup{d}_x f (k/n)<0$. Therefore the slow-divergence integral on the intermediate section of the consensus space is null, so on the intermediate section the attracting and repelling effect compensate each other.
The considerations we have done hold for the dynamics near consensus under a constant critical perturbation. It is worth mentioning that the nonlocal interaction of transcritical singularities via hysteresis processes could give rise to an enhanced delay for the canard \cite{10.36045/bbms/1228486410,doi:10.1142/S0218127412500265}.


\subsection{Spatio-temporal configurations}\label{sec:spatio-temporal}

In this section, we consider the spatio-temporal configurations induced by
different network structures both under generic and critical perturbations, the
latter giving rise to canards. The simulations are not subject to any reduction, and the evolution of the full system is computed. Given a graph we consider a planar embedding,
and we assume the obtained coordinates as representative of a portion of a
square. The spatio-temporal configurations are then obtained by evaluating a
heat map where the colors, from blue to yellow, are in correspondence with the
state's values of the system. Each frame represents a unit time step in the
sampled interval, and the time flow is represented as follows
\begin{equation}
    \begin{matrix}
        t_1    & t_{n+1} & \cdots \\
        t_2    & t_{n+2} & \cdots \\
        \vdots & \vdots  & \vdots \\
        t_n    & t_{n+n} & \cdots
    \end{matrix} .
\end{equation}

We considered the following network structures: random networks (Figures
\ref{fig:random_generic} and \ref{fig:random_canard}), a complete graph with
$10$ nodes (Figure \ref{fig:complete}), and a lattice with $30 \times 30$
nodes (Figures \ref{fig:lattice_timeseries}--\ref{fig:lattice_30_canard_stp}).
Throughout this section we adopt the response function $f(x)=(x-1)^2(x+1)^2$.

Spatio-temporal patterns are usually considered in the realm of partial
differential equations
\cite{doi:10.1142/9789813239609_0004,doi:10.1137/19M1306610,PAGE200595} where a wide range of equations, e.g., reaction-diffusion systems, can give rise to complex patterns observable in
nature. Here, the continuous space is replaced by a network and consequently in
place of partial differential equations we have a set of ordinary differential
equations. Given the analogy with diffusion and drift mechanisms that perturbed
ALFs undergo it is quite natural to explore the potential spatio-temporal
patterns appearing under different conditions. The most remarkable observation
we can draw from the simulations is that under generic perturbations no
recognizable pattern appear, while under a critical perturbation we observe
transient patterns. In other words, our simulations suggest that \emph{spatio-temporal patterns are induced
    by canards}. Such an effect is related to the delayed loss of stability
produced by a canard. In fact, we generically have that the individual states,
once they reach the singularity, leave consensus one after the other, with no
particular order. On the other hand, under a critical perturbation a
significant amount of states transition at the same time and such a
configuration persists for several units of time. Such phenomenon is more
clearly visible when the number of nodes is large, to see this compare the
plots for a complete graph with $10$ nodes (Figure
\ref{fig:complete_canard_stp}), a random graph with $30$ nodes (Figure
\ref{fig:random_canard_stp}) and a lattice with $900$ nodes (Figure
\ref{fig:lattice_30_canard_stp}).

\begin{figure}[htbp]
	\begin{subfigure}[T]{.5\textwidth}
		\centering
		\includegraphics[width=.8\linewidth]{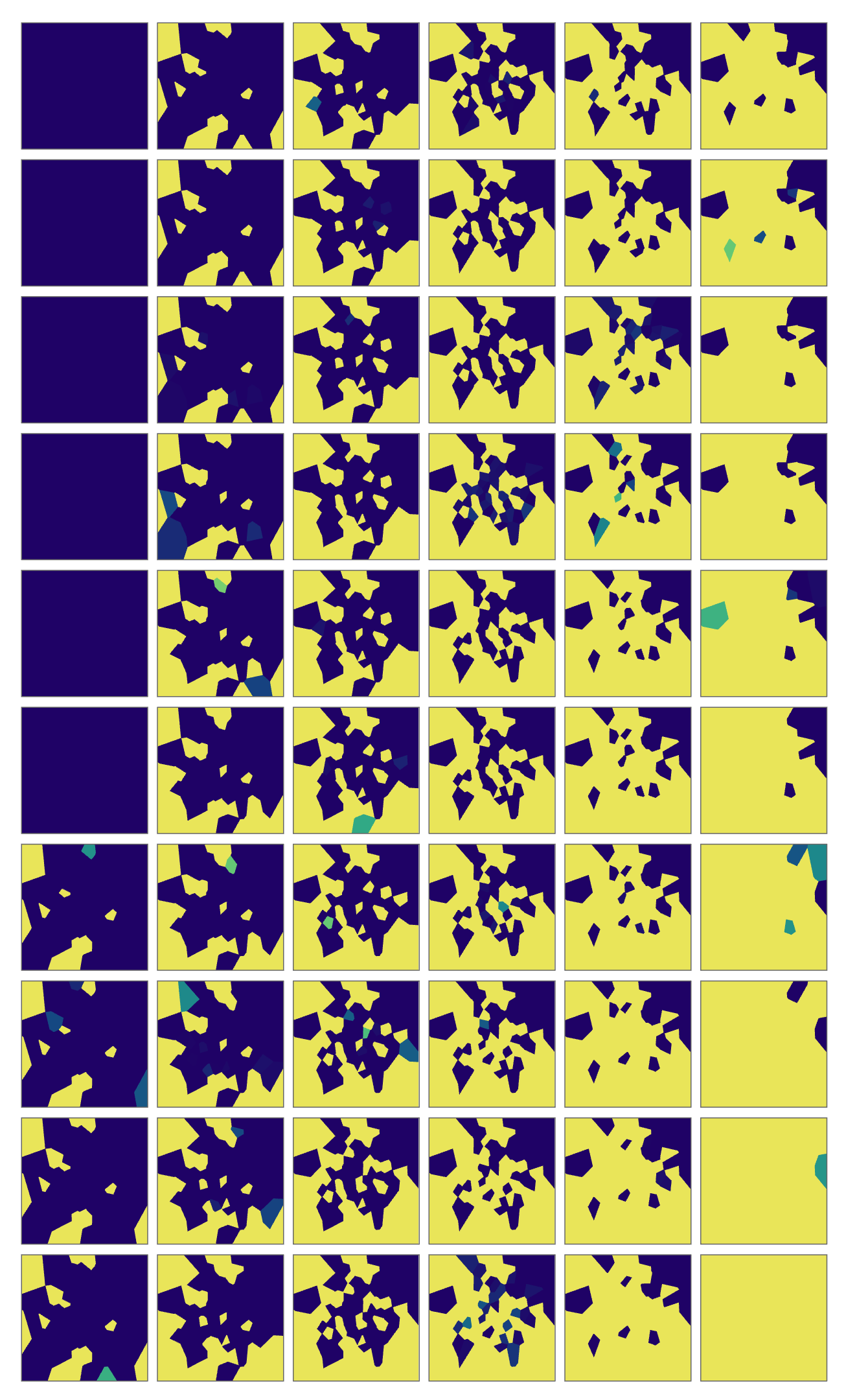}
		\caption{}
		\label{fig:random_stp}
	\end{subfigure}
	\begin{subfigure}[T]{.5\textwidth}
		\centering
		\includegraphics[width=.8\linewidth]{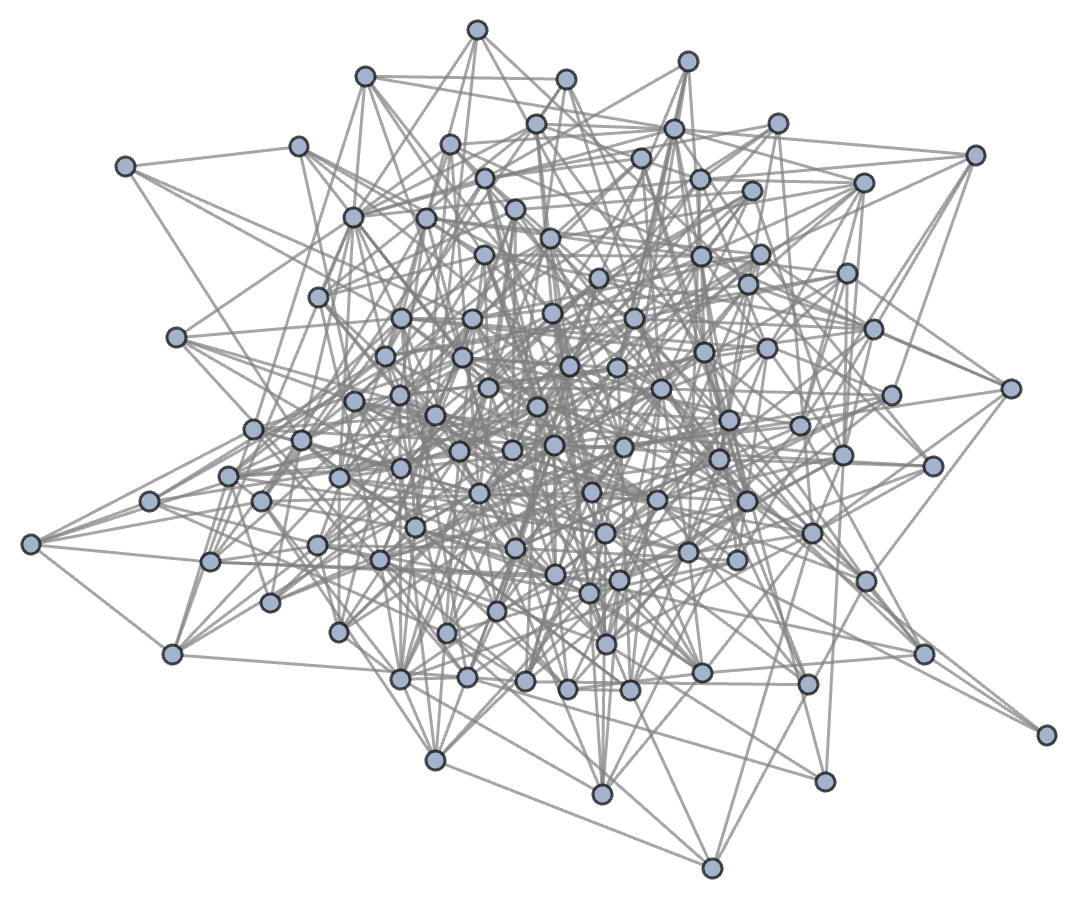}
		\caption{}
		\label{fig:random_graph_generic}
		\centering
		\includegraphics[width=.8\linewidth]{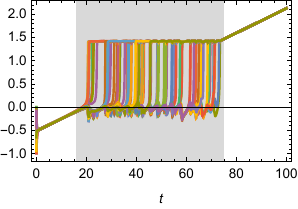}
		\caption{}
		\label{fig:random_ts}
	\end{subfigure}
	\caption{Spatio-temporal configurations (Figure \ref{fig:random_stp}) for a random graph with $100$ nodes and $500$ edges (Figure \ref{fig:random_graph_generic}), for $\epsilon=1/20$. The perturbation is positive, each component is randomly chosen in $[0,1]$ with uniform distribution. The interval of sampling is displayed as the gray region in the time series \ref{fig:random_ts}; the sampling rate is one unit of time. The initial conditions for each variable are randomly chosen in the interval $[-1,0]$ with uniform distribution. The perturbation induces a drift of the system to higher values, such drift is clearly visible in the time series \ref{fig:random_ts}. The heat map range goes from blue (lower values) to yellow (higher values), passing through green.}
	\label{fig:random_generic}
\end{figure}

\begin{figure}[htbp]
	\centering
	\begin{subfigure}[T]{0.35\textwidth}
		\centering
		\includegraphics[width=0.7\linewidth]{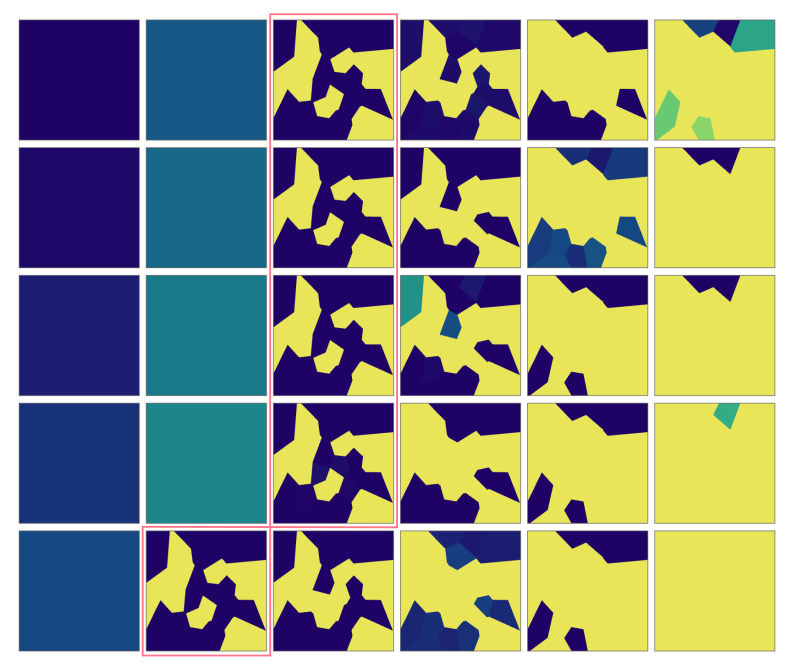}
		\caption{}
		\label{fig:random_canard_stp}
	\end{subfigure}
	\begin{subfigure}[T]{0.35\textwidth}
		\centering
		\includegraphics[width=0.8\linewidth]{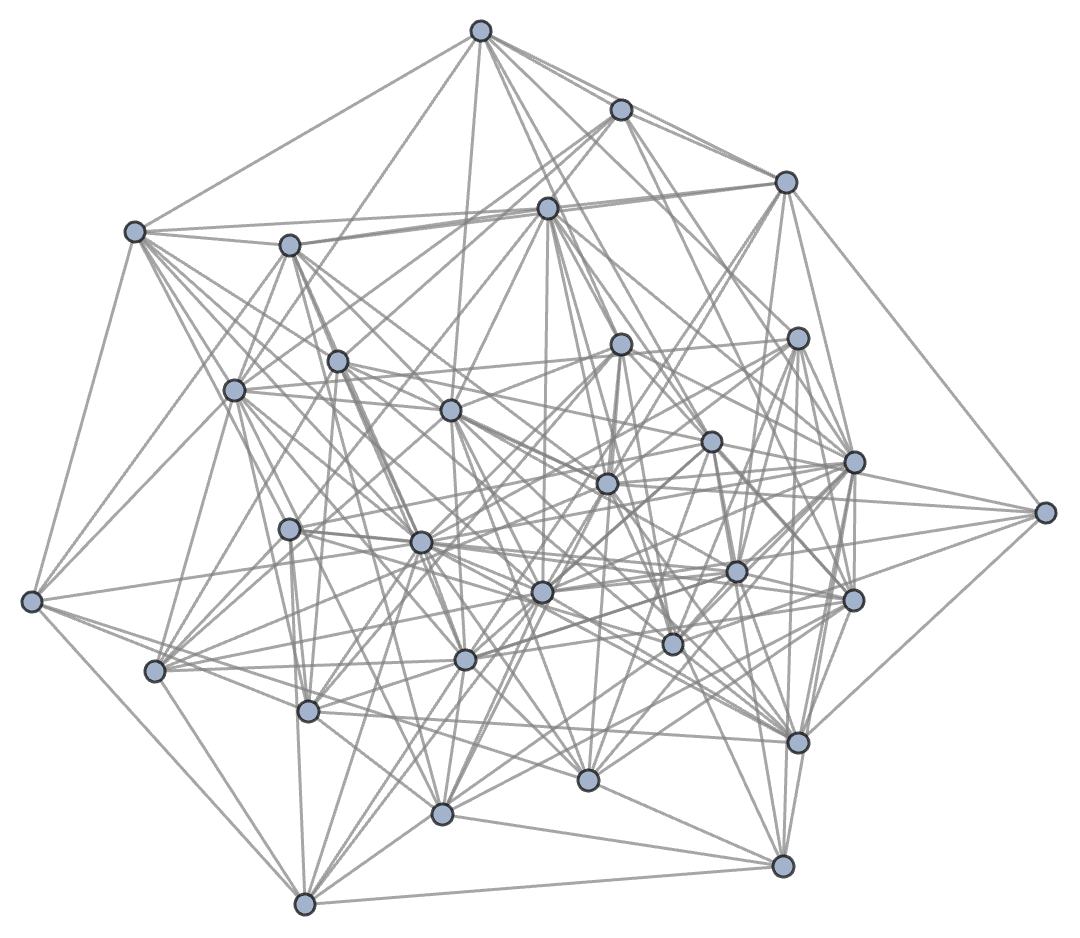}
		\caption{}
		\label{fig:graph_random_canard}
	\end{subfigure}
	\begin{subfigure}[T]{0.35\textwidth}
		\centering
		\includegraphics[width=0.9\linewidth]{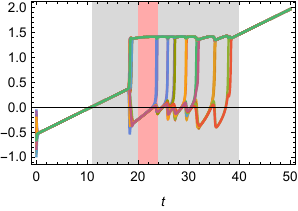}
		\caption{}
		\label{fig:random_canard_ts}
	\end{subfigure}
	\caption{Spatio-temporal configurations (Figure \ref{fig:random_canard_stp}) for a random graph with $30$ nodes and $150$ edges (Figure \ref{fig:graph_random_canard}) under a positive critical perturbation, $H^\textup{crit}=\mathbf{1}$, $\epsilon=1/20$. The interval of sampling is displayed as the gray region in the time series \ref{fig:random_canard_ts}, the sampling rate is one unit of time. In Figure \ref{fig:random_canard_ts} we highlighted in pink the region where a persistent pattern induced by a canard is visible in the maps \ref{fig:random_canard_stp}. The simulation is performed with a working precision of $50$ digits.}
	\label{fig:random_canard}
\end{figure}

\begin{figure}[htbp]
	\centering
	\begin{subfigure}[T]{0.45\textwidth}
		\centering
		\includegraphics[width=0.6\linewidth]{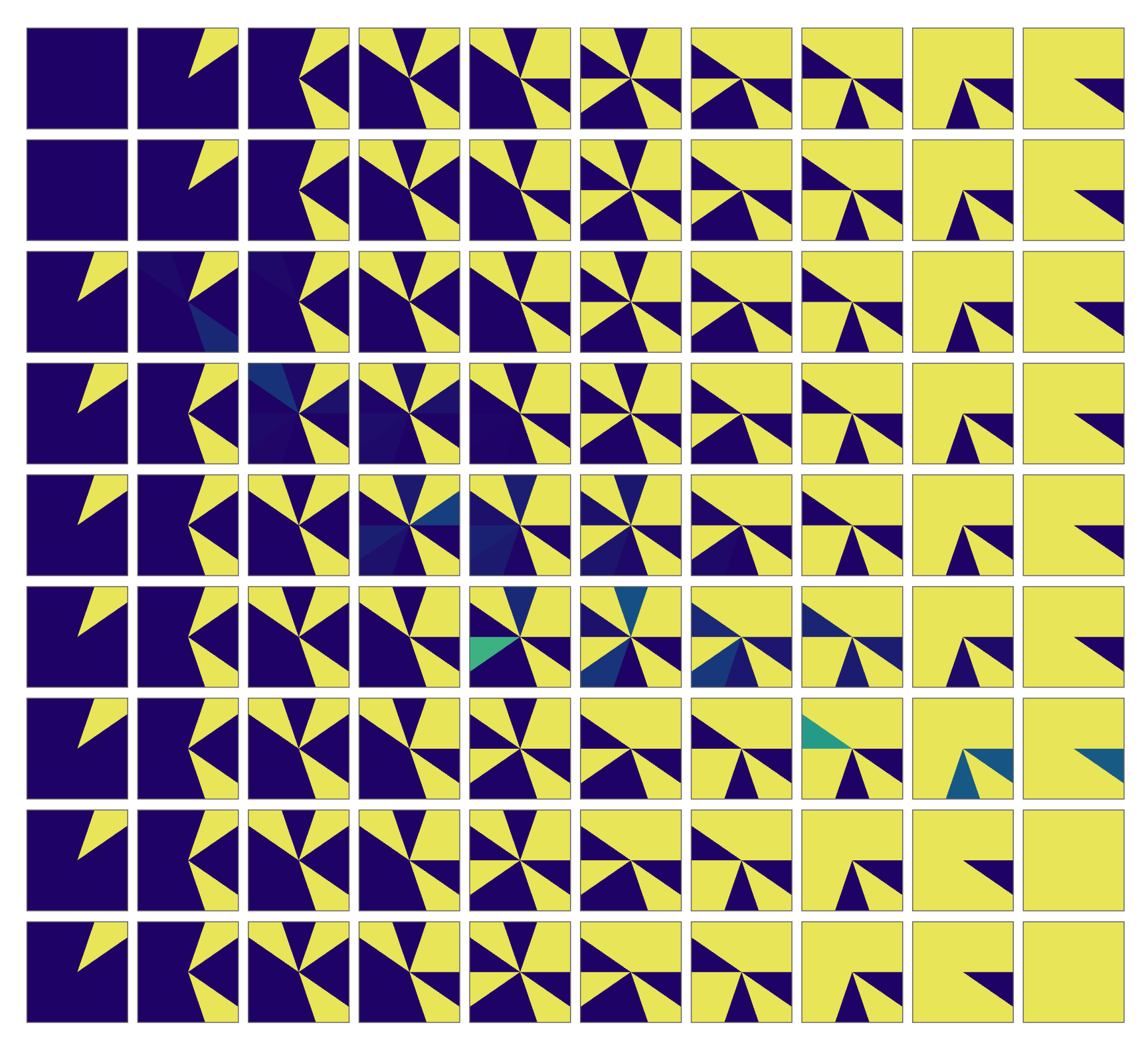}
		\caption{}
		\label{fig:complete_generic_stp}
	\end{subfigure}
	\hspace{0.05\textwidth}
	\begin{subfigure}[T]{0.45\textwidth}
		\centering
		\includegraphics[width=0.75\linewidth]{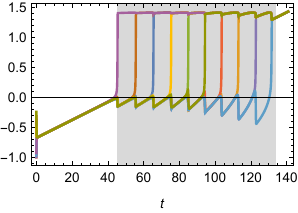}
		\caption{}
		\label{fig:complete_generic_ts}
	\end{subfigure}
	\vskip 0.5cm
	\begin{subfigure}[T]{0.45\textwidth}
		\centering
		\includegraphics[width=0.65\linewidth]{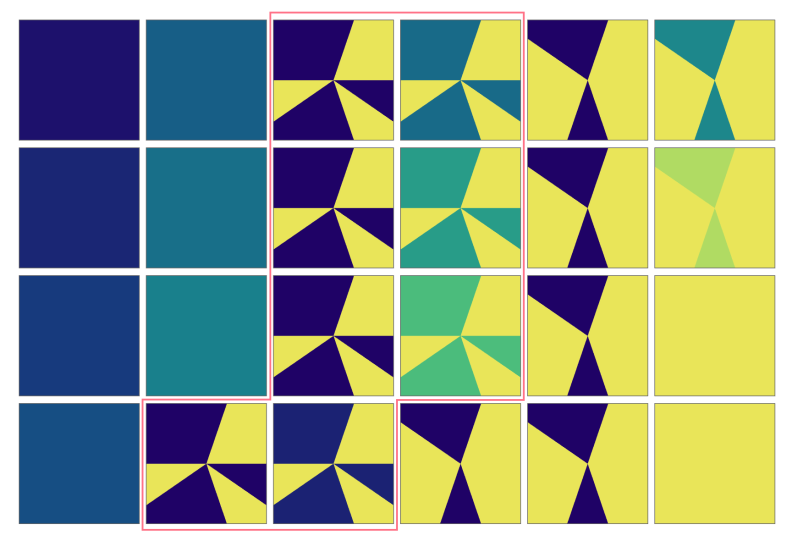}
		\caption{}
		\label{fig:complete_canard_stp}
	\end{subfigure}
	\hspace{0.05\textwidth}
	\begin{subfigure}[T]{0.45\textwidth}
		\centering
		\includegraphics[width=0.75\linewidth]{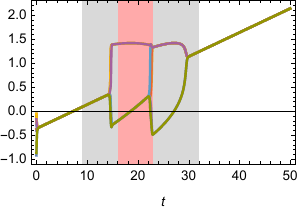}
		\caption{}
		\label{fig:complete_canard_ts}
	\end{subfigure}
	\caption{Spatio-temporal configurations for a complete graph with $10$ nodes for $\epsilon=1/20$. Figure \ref{fig:complete_generic_stp} and \ref{fig:complete_generic_ts} are obtained with a constant perturbation, where each component is randomly chosen in $[0,1]$ with uniform distribution, while Figure \ref{fig:complete_canard_stp} and \ref{fig:complete_canard_ts} are obtained using a positive critical perturbation, $H^\textup{crit}=\mathbf{1}$. The initial conditions for each variable are randomly chosen in the interval $[-1,0]$ with uniform distribution. The interval of sampling is given by the gray area in the time series; the sampling rate is one unit of time. The pink region highlights the persistent pattern induced by the canard. The simulations are performed with a working precision of $50$ digits.}
	\label{fig:complete}
\end{figure}

\begin{figure}[htbp]
	\centering
	\begin{subfigure}[T]{0.4\textwidth}
		\includegraphics[width=\linewidth]{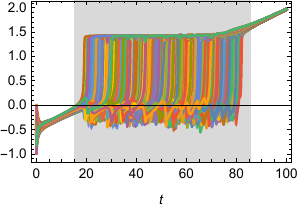}
		\caption{}
		\label{fig:lattice_30_generic_ts}
	\end{subfigure}
	\begin{subfigure}[T]{0.4\textwidth}
		\includegraphics[width=\linewidth]{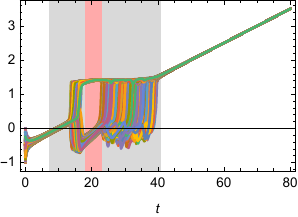}
		\caption{}
		\label{fig:lattice_30_canard_ts}
	\end{subfigure}
	\caption{Time series for a lattice with $30 \times 30$ nodes and $\epsilon=1/20$. For Figure \ref{fig:lattice_30_generic_ts} each component of the perturbation is randomly chosen in $[0,1]$ with uniform distribution, while for Figure \ref{fig:lattice_30_canard_ts} we used a positive critical perturbation, $H^{crit}=\mathbf{1}$. The corresponding spatio-temporal configurations are displayed respectively in Figure \ref{fig:lattice_30_generic_stp} and \ref{fig:lattice_30_canard_stp}. The interval of sampling is given by the gray area in the time series; the sampling rate is one unit of time. The pink region highlights the persistent pattern induced by the canard. The initial conditions for each variable are randomly chosen in the interval $[-1,0]$ with uniform distribution. The simulations are performed with a working precision of $20$ digits.}
	\label{fig:lattice_timeseries}
\end{figure}

\begin{figure}[htbp]
	\centering
	\includegraphics[width=.5\linewidth]{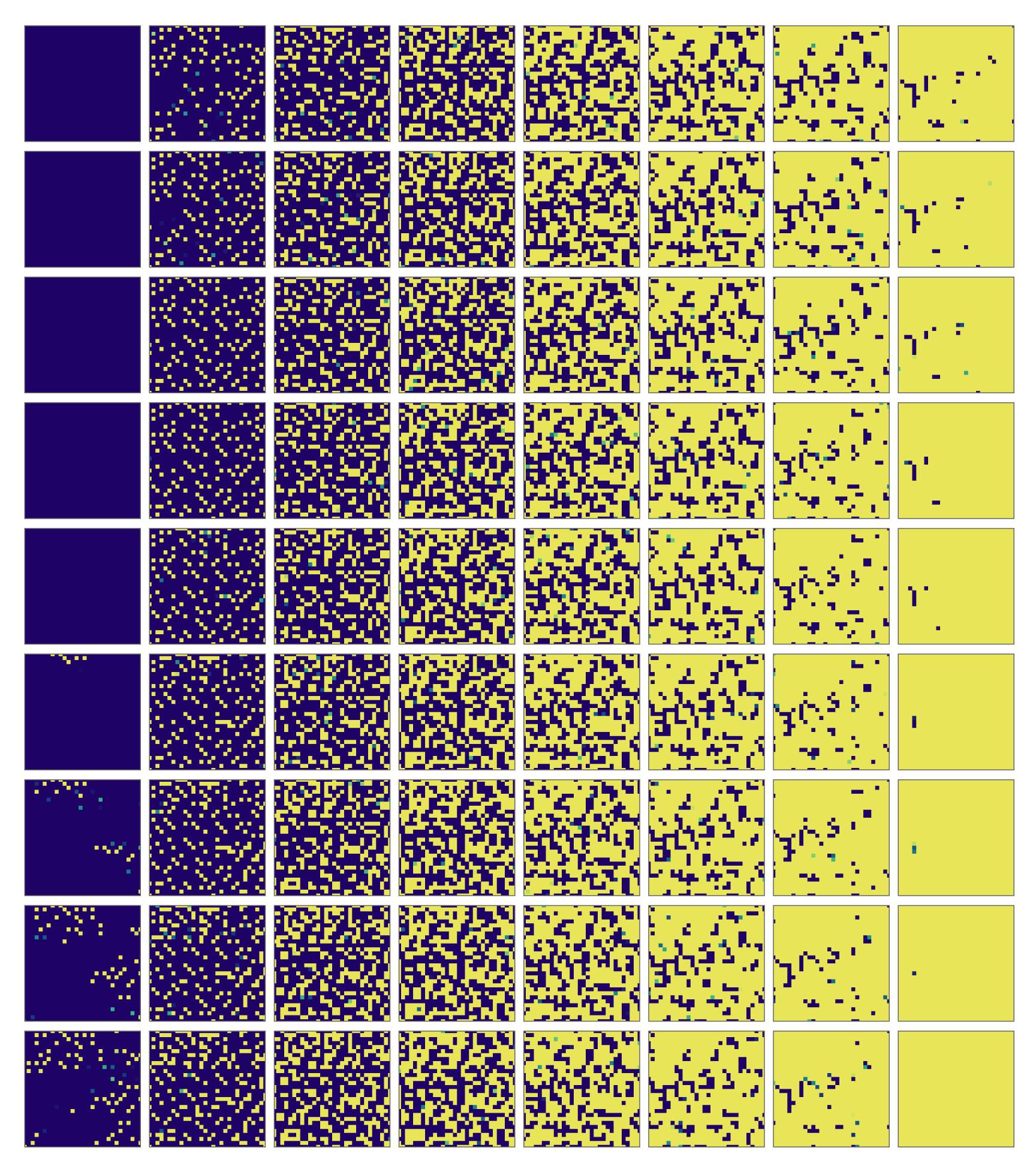}
	\caption{Spatio-temporal configurations for a lattice with $30 \times 30$ nodes under a positive random constant perturbation. The drift process from lower values (blue) to higher values (yellow) occurs by fast transitions, or jumps, due to the intricate nature of the critical manifold, which in turn is induced by the nonlinearity of the response function.}
	\label{fig:lattice_30_generic_stp}
\end{figure}

It is worth mentioning that spatio-temporal patterns are commonly studied and understood for reaction-diffusion equations where the two species interaction could give rise to patterns. Here, we are studying equations that resemble a drift-diffusion, and the (transient) pattern structure arises from the nonlinearity of the diffusion mechanism and the geometry of the drift.

In general, we can notice that the nonlinearity given by the response function
induces a different diffusion mechanism. By looking at the generic cases,
illustrated in Figures \ref{fig:random_stp},
\ref{fig:complete_generic_stp}, \ref{fig:lattice_30_generic_stp}, we see that
the transition between the initial consensus state to the final one exhibits
jumps instead of a homogeneous process. Such an effect is due to the fact that
the nonlinearity induces an intricate structure of the critical manifold,
as we have already seen in the previous examples. We can partially observe the
homogeneous process of diffusion when the canard conditions apply, see
Figures \ref{fig:random_canard_stp}, \ref{fig:complete_canard_stp} (for
the lattice in Figure \ref{fig:lattice_30_canard_stp} such an effect is suppressed
by the image quality compression).

As a final remark, we observe that the transition time between two consensus
sections is visibly shorter in the canard cases.

\begin{figure}[htbp]
    \centering
    \includegraphics[width=.5\linewidth]{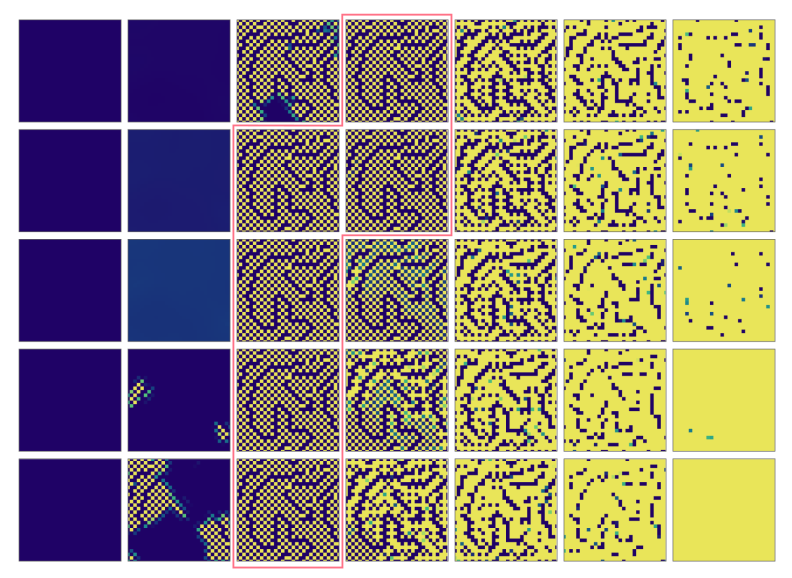}
    \caption{Spatio-temporal configurations for a lattice with $30 \times 30$ nodes under a positive critical perturbation. In this case, it is particularly evident the formation of a pattern, persisting several units of time, induced by the delayed loss of stability caused by the canard.}
    \label{fig:lattice_30_canard_stp}
\end{figure}



\section{Conclusions and discussion }\label{sec:conclusions}

This paper has been dedicated to the analysis of a class of dynamical
    systems describing nonlinear diffusion on a simple graph structure. The
analysis starts by looking at (unperturbed) ALFs on graphs with positive
weights. We emphasize the importance of the fiber-wise topological
equivalence proved in Proposition \ref{prop:equivalence}. Indeed, a
weighted graph, in general, loses the symmetry properties that its own
unweighted counterpart has. The fact that a positive-weighted graph induces
an evolution of the system that is topologically equivalent to the
one induced by the unweighted counterpart, allows us to consider unweighted
simple graphs. An ALF with a simple graph structure inherits the
symmetry properties of the underlying graph, enabling us to use the tools
from equivariant dynamical systems theory, such as fixed-point spaces and
    invariant properties of the response function. The structure of the
    equations allows for a detailed analysis of the equilibria, that in turn led us
    to study the local structure of singularities on the consensus space for
    complete graphs (Propositions \ref{prop:tangents} and \ref{prop:canard}), and
    to prove the existence of a maximal canard for generic simple graphs
    (Proposition \ref{prop:canard_generic}).

Several new directions arise from the results of our paper. For example, one
could consider response functions $f(\bm{x})$ depending also on
\enquote{neighboring} nodes; in such a case the network structure also appears
inside the response vector field, and not only through the Laplacian matrix.
Moreover, one may be interested in studying the case of heterogeneous response
functions, i.e., different nodes might contribute differently to the mutual
interaction. Regarding the graph structure, one may take into account negative
weights or directed graphs, which, even for the linear case, give rise to much richer dynamics than
the positive-weights case. Hypergraphs are also becoming more and more relevant. Incorporating a hypergraph structure in our model could take different directions. The simplest choice would be to study which class of hypergraph models can be associated with an ALF. Of course, such a choice would lead unavoidably to a restricted family of hypergraphs. On the other hand, one could extend the model to hypergraphs, by replacing the graph Laplacian with the hypergraph Laplacian, and the other quantities accordingly. In this case, the system will differ dramatically from the one we studied. However, our analysis could serve as a guideline: Exploiting symmetries and algebraic properties could lead to relevant simplifications for the study.
In a similar line of thought, reduction techniques similar to those presented here, but
for other network topologies than complete networks, seems worth pursuing.
Another possibility is to study Laplacian systems where the states have
dimensions larger than one, for example, one might consider Laplacian systems
where the states evolve in the real plane, or even in a two-dimensional
manifold. The dynamics of planar systems entail interesting phenomena, such as
periodic orbits, that could potentially give rise to new phenomena unseen in
the one-dimensional framework. For similar reasons, states of dimension three
might be relevant as well, since it is well-known that it is for
three-dimensional systems that chaos, at the level of the nodes, may appear. In
all the aforementioned cases, a natural question is to investigate, and
classify the possible singularities that arise and the influence of the
network on them (for example, in this paper, we showed that complete ALFs can
only have transcritical singularities along the consensus manifold). Moreover,
one would then be interested in the behavior under small perturbations, which,
as seen here, is also closely related to the network structure. Finally, given its relevance for network dynamics, it would be interesting to further investigate which perturbations lead to persistence, or on the other hand destroy, the singular canards of Proposition \ref{prop:canard}. One possibility is already given in Proposition \ref{prop:canard_generic}, but its generalization seems rather promising.


\vspace{6pt}
\appendix


\section{Groups}\label{app:sym}
We recall some basic notions about group theory
\cite{rotman1999introduction,harris1991representation}.

\begin{Definition}
    Let $\Gamma$ be a set and $\star: \Gamma \times \Gamma \to \Gamma$ a binary operation. A \emph{group} is a couple $(\Gamma, \star)$ satisfying the following properties
    \begin{enumerate}
        \item Associativity: $(\gamma_i \star \gamma_j) \star \gamma_k = \gamma_i \star
                  (\gamma_j \star \gamma_k)$, $\forall \gamma_i, \gamma_j, \gamma_k \in \Gamma$.
        \item Identity: $\exists \ e \in \Gamma$ such that $e \star \gamma = \gamma \star e =
                  \gamma$, $\forall \gamma \in \Gamma$.
        \item Inverse: $\forall \gamma \in \Gamma$, $\exists \ \gamma^{-1} \in \Gamma$ such
              that $\gamma^{-1} \star \gamma = \gamma \star \gamma^{-1} = e$.
    \end{enumerate}
\end{Definition}

The operation $\star$ is called group operation. Note that, closure under the
group operation follows from the definition of the operation itself. Unless
strictly necessary, we refer to a group just by the set, e.g., $\Gamma$,
implicitly considering it equipped with a group operation; we also omit the
group operation symbol, e.g., $\gamma_i \gamma_j = \gamma_i \star \gamma_j$.

\begin{Definition}
    A group is said to be \emph{finite} if the set $\Gamma$ has a finite number of elements. The \emph{order} of a finite group, $\textup{Ord}(\Gamma)$, is the number of elements in the group.
\end{Definition}

Given a vector space, a representation, $\psi$, is a homomorphism from the
group to the general linear group of the vector space. Since the states,
$\bm{x}$, of an ALF belong to $\mathbb{R}^n$, we are interested in how
$\Gamma$ acts on $\mathbb{R}^n$, therefore, we consider representations $\psi :
    \Gamma \to GL(\mathbb{R}^n)$. To avoid new unnecessary notation, we omit to
explicitly write the representation. It is implicit that when we write a group
element applied to a vector we are considering the representation on the
appropriate vector space, e.g., $\gamma \bm{x} = \psi(\gamma) \bm{x}$.

\begin{Definition}
    Given two groups $\Gamma^{(1)}$, $\Gamma^{(2)}$, the \emph{direct product} $\Gamma^{(1)} \times \Gamma^{(2)}$ is a group with set of elements given by the Cartesian product
    \begin{equation}
        \Gamma^{(1)} \times \Gamma^{(2)} = \left\{ (\gamma^{(1)}, \gamma^{(2)}) \ | \ \gamma^{(1)} \in \Gamma^{(1)} ,  \gamma^{(2)} \in \Gamma^{(2)} \right\} ,
    \end{equation}
    and operation defined as follows
    \begin{equation}
        (\gamma^{(1)}_i, \gamma^{(2)}_j) (\gamma^{(1)}_k, \gamma^{(2)}_l) = (\gamma^{(1)}_i \gamma^{(1)}_k, \gamma^{(2)}_j \gamma^{(2)}_l) ,
    \end{equation}
    $\forall \gamma_{i}^{(1)}, \gamma_{k}^{(1)} \in \Gamma^{(1)}$ and $\forall \gamma_{j}^{(2)}, \gamma_{l}^{(2)} \in \Gamma^{(2)}$.
\end{Definition}

It follows that the direct product can be extended to an arbitrary number of
groups, $\Gamma^{(1)} \times \Gamma^{(2)} \times \dots \times \Gamma^{(n)}$.
Another important concept is the one of \emph{fixed-point space}, i.e., the set
of points that are kept fixed by the action of the group.

\begin{Definition}\label{def:fix}
    Let $\Sigma \subseteq \Gamma$ be a subgroup of $\Gamma$. The \emph{fixed-point space} (under the action in $\mathbb{R}^n$) of the group $\Sigma$ is denoted by $\fix(\Sigma)$ and it is given by
    \begin{equation}
        \fix(\Sigma) := \left\{ \bm{x} \in \mathbb{R}^n \ | \  \sigma \bm{x}=\bm{x} ,\ \forall \sigma \in \Sigma \right\} .
    \end{equation}
\end{Definition}

A direct interplay between graph theory and group theory is given by the
graph's automorphism group. For graphs the automorphism group is a subgroup of
the (finite dimensional) symmetric group $\mathfrak{S}_n$
\cite{sagan2001symmetric}, i.e., the group of permutations of $n$ symbols.

\begin{Definition}\label{def:aut_g}
    Given a (simple) graph $\mathcal{G}$, the \emph{automorphism group}, $\textup{Aut}(\mathcal{G})$, is the group of permutations which preserves the adjacency structure of the vertices of $\mathcal{G}$.
\end{Definition}

\begin{Remark}
    From an algebraic point of view, preservation of adjacency structure correspond to the fact that $\sigma \in \textup{Aut}(\mathcal{G})$  if and only if $[ A,  \sigma ] =0$, where $[ \cdot ,\cdot ]$ represent the commutator of matrices \cite{biggs1993algebraic}. Notice that the same statement holds also if we substitute the adjacency matrix with the Laplacian matrix.
\end{Remark}

\subsection{Invariance and equivariance}

\begin{Definition}
    Let $\Gamma$ be a group\footnote{Notice that this is not related to $\textup{Aut}(\mathcal{G})$, we use the same symbol $\Gamma$ for different groups in different context.}. A function $f$ is called \emph{$\Gamma$-invariant}  if
    \begin{equation}
        f \circ \Gamma  = f ,
    \end{equation}
    where $\circ$ denotes the composition of maps.
\end{Definition}

The concept of symmetry, in dynamical systems theory, refers to the fact that
there exists a group mapping solutions to solutions. When a vector field is
given, such a concept can be translated in terms of equivariance of the vector
field.

\begin{Definition}
    Let $\Gamma$ be a group, and $X$ a vector field. Then $X$ is called \emph{$\Gamma$-equivariant} if
    \begin{equation}
        X \circ \Gamma = \Gamma \circ X .
    \end{equation}
    An element of the group, $\gamma \in \Gamma$, is called a symmetry of the vector field
\end{Definition}

\begin{Remark}
    It is straightforward to see that if a system is $\Gamma$-equivariant, then $\Gamma$ maps solutions to solutions.
\end{Remark}

A solution of a dynamical system is a simple example of an invariant set under
the flow. Invariant sets, in general, are extremely relevant in the study of
ODEs, since they allow to reduce the analysis to potentially simpler equations.
We recall here the general definition of invariant set under the flow induced
by a vector field. Later we state the interplay between symmetries and
invariant sets.

\begin{Definition}
    Let $\phi$ be the flow generated by a vector field on $\mathbb{R}^n$. A set $\mathcal{M} \subset \mathbb{R}^n$ is said to be invariant (respectively positively/negatively invariant) under the flow if for every $x\in\mathcal M$, $\phi_t(x)\in\mathcal M$, $\forall t \in \mathbb{R}$ (respectively $\forall t \in \mathbb{R}_{+} / \mathbb{R}_{-}$).
\end{Definition}

In order to keep a compact notation, we call a set invariant under the flow a
$\phi$-invariant set. From equivariance of the vector field we can immediately
infer the existence of a $\phi$-invariant set.

\begin{Proposition}\label{prop:invariant_space}
    The fixed-point space, $\fix(\Gamma)$, is a $\phi$-invariant set for $\Gamma$-equivariant systems \cite{golubitsky2012singularities}.
\end{Proposition}



\section{Singular perturbations}\label{app:singular}

In general, the symbols used in this section are not related to those in the
main text.

A singularly perturbed system of ODEs, in the standard form, is also called a
\emph{slow-fast system}, and it has the form
\begin{equation}\label{eq:standard_fast}
    \begin{aligned}
        \dot{x} & = f(x,y,\epsilon)            \\
        \dot{y} & = \epsilon g(x,y,\epsilon) ,
    \end{aligned}
\end{equation}
where $0<\epsilon \ll 1$, $x \in \mathbb{R}^n$, $y \in \mathbb{R}^m$, and the time-parametrisation $t$ is called \emph{fast-time}. Equation \eqref{eq:standard_fast} is said to be in the standard form because the time separation between the fast-variables, $x$, and the slow-variables, $y$, is explicit. In the non-standard form a singularly perturbed system has the form
\begin{equation}
    \dot{z} = F(z, \epsilon) ,
\end{equation}
where $z \in \mathbb{R}^{n+m}$, and the vector field $F(z, 0)$ has at least one set of equilibria of dimension greater than zero \cite{wechselberger2020geometric}.

Defining the \emph{slow-time}, $\tau:= \epsilon t$, equation
\eqref{eq:standard_fast} can be rewritten as
\begin{equation}\label{eq:standard_slow}
    \begin{aligned}
        \epsilon  x' & = f(x,y,\epsilon)    \\
        y'           & =  g(x,y,\epsilon) .
    \end{aligned}
\end{equation}
Setting $\epsilon$ to zero in equation \eqref{eq:standard_slow}, we obtain the
\emph{reduced system}
\begin{equation}\label{eq:standard_reduced}
    \begin{aligned}
        0  & = f(x,y,0)    \\
        y' & =  g(x,y,0) ,
    \end{aligned}
\end{equation}
which is a set of constrained differential equations.
Instead, setting $\epsilon$ to zero in equation \eqref{eq:standard_fast} we obtain the \emph{layer problem}
\begin{equation}\label{eq:standard_layer}
    \begin{aligned}
        \dot{x} & = f(x,y,0) \\
        \dot{y} & = 0 ,
    \end{aligned}
\end{equation}
where the slow-variables, $y$, play the role of parameters.

\begin{Definition}
    The \emph{critical manifold} is defined as the set
    \begin{equation}
        \mathcal{C}_0 = \{ (x,y) \in \mathbb{R}^n \times \mathbb{R}^m \ | \ f(x,y,0) = 0 \}.
    \end{equation}
\end{Definition}

\begin{Remark}
    Regardless of the commonly accepted name, the critical manifold is not necessarily a manifold. 
\end{Remark}
The critical manifold is the constraint space for the reduced system \eqref{eq:standard_reduced}, or equivalently, it is the set of equilibria of the layer problem \eqref{eq:standard_layer}.

\begin{Definition}
    A point $x^* \in \mathcal{C}_0$ is called \emph{normally hyperbolic} if the Jacobian matrix $\textup{D}_x f(x^*,y,0)$, associated to the layer problem \eqref{eq:standard_layer}, has no eigenvalues on the imaginary axis.
\end{Definition}

\begin{Definition}
    A submanifold of the critical manifold $\mathcal{M}_0 \subseteq \mathcal{C}_0$ is called normally hyperbolic if all points in $\mathcal{M}_0$ are normally hyperbolic. Moreover, invariant manifolds that are normally hyperbolic are called normally hyperbolic invariant manifold (NHIM) \cite{hirsch2006invariant}.
\end{Definition}

\begin{Theorem}[Fenichel]\label{thm:fenichel}
    For $\epsilon$ sufficiently small, compact normally hyperbolic submanifolds of the critical manifold, $\mathcal{M}_0 \subseteq \mathcal{C}_0$, persists as $O(\epsilon)$-close locally invariant slow manifolds $\mathcal{M}_\epsilon$ diffeomorphic to $\mathcal{M}_0$. Moreover, $\mathcal{M}_\epsilon$ is normally hyperbolic and has the same stability properties with respect to the fast variables as $\mathcal{M}_0$.
\end{Theorem}

\begin{Remark}
    Theorem \ref{thm:fenichel} is restricted to the results relevant for the present paper. An exhaustive version can be found in \cite{fenichel1979geometric, kuehn2015multiple}. The terminology $O(\epsilon)$-close means that, as $\epsilon \to 0$, $\mathcal{M}_\epsilon$ has Hausdorff distance $O(\epsilon)$ from $\mathcal{M}_0$.
\end{Remark}

For nonhyperbolic points, $x^* \in \mathcal{C}_0$, also known as \emph{singularities}, Fenichel's theorem does not provide information about the perturbed system in a neighborhood of $x^*$. The process of analysis of the system in a neighborhood of a singularity is called desingularisation. We briefly introduce the \emph{blow-up} (desingularisation) for \emph{nilpotent} singularities \cite{https://doi.org/10.48550/arxiv.1901.01402}.

\begin{Definition}
    A point $x^* \in \mathcal{C}_0$ is a \emph{nilpotent singularity} if all the eigenvalues of the Jacobian, at $x^*$, are zero.
\end{Definition}

When a singular point has a mixture of zero eigenvalues, and non-zero (real
part) eigenvalues it is possible to reduce the problem to a nilpotent
singularity via a \emph{center manifold} reduction \cite{carr1979applications}.

\begin{Definition}

    Let $x^* $ be a nilpotent singularity. The \emph{spherical blow-up} is a local transformation
    \begin{equation}\label{eq:blowup}
        \begin{aligned}
            \beta : \mathbb{S}^{n+m}_{x^*}   \times I  & \to    \mathbb{R}^{n+m+1}         \\
              (\bar{x}, \bar{y},  \bar{\epsilon}, r)       & \mapsto  (x, y, \epsilon )  ,
        \end{aligned}
    \end{equation}
    where $\mathbb{S}^{n+m}_{x^*}$ is the $n+m$-dimensional sphere, embedded in $ \mathbb{R}^{n+m+1}$, centred at $x^*$, $(\bar{x}, \bar{y},  \bar{\epsilon}) \in \mathbb{S}^{n+m}_{x^*}$, $I \subseteq \mathbb{R}_{\geq 0}$ is an interval such that $0 \in I$, and $r \in I$.
\end{Definition}

Let us notice that, for $r>0$ the transformation
\eqref{eq:blowup} is a diffeomorphism, while for $r=0$ it is not invertible. Roughly speaking, the
blow-up maps the point $x^*$ into a $(n+m)$-dimensional sphere. The core
idea of the blow-up is to analyze the induced flow on the sphere and deduct the
behavior of the singularity. However, the blow-up transformation by itself
does not desingularise the system. After the blow-up it is necessary to perform
a conformal transformation (time reparametrization) in order to regularize the
system, i.e., retrieve hyperbolicity. We described the spherical blow-up, but
several other blow-up transformations are possible, e.g., quasihomogeneous,
directional, or even to other manifolds, not necessarily spheres
\cite{kuehn2016remark}. The choice of the appropriate blow-up should be done by
considering the regularization, and the feasibility of the analysis. Since this
work is not concerned with a direct application of the blow-up, but rather with
the implementation of results involving the blow-up technique, we presented the
spherical set-up which provides the most clear geometrical visualization.



%
%

\bibliographystyle{spmpsci}      
\bibliography{bibliography.bib}

%
%
\section*{Statements and Declarations}

\begin{description}
    \item[\textbf{Funding}] The authors declare that no funds, grants, or other support were received during the preparation of this manuscript.
    \item[\textbf{Competing Interests}] The authors have no relevant financial or non-financial interests to disclose.
    \item[\textbf{Data Availability}] All datasets are available upon request to the corresponding author.
\end{description}

\end{document}